\documentclass[10pt,reqno]{amsart}
\usepackage{amsmath,amsthm,amsfonts,amscd,amssymb,latexsym,mathrsfs}
\usepackage{comment}
\usepackage{stmaryrd}
\usepackage{graphicx}
\usepackage[pdfusetitle,colorlinks]{hyperref}
\PassOptionsToPackage{dvipsnames}{xcolor}
\hypersetup{citecolor=OliveGreen,linkcolor=Mahogany,urlcolor=Plum}
\usepackage[utf8]{inputenc}
\usepackage[capitalise]{cleveref}
\usepackage{enumerate}
\usepackage[all,cmtip]{xy}
\usepackage{tikz}
\usepackage{float}
\usetikzlibrary{arrows}
\usetikzlibrary{patterns}
\newtheorem*{question}{Question}
\numberwithin{equation}{subsection}
\newtheoremstyle{plain2}    
   {}            
   {}            
   {\itshape}    
   {}            
   {\bfseries}   
   {.}           
   {5pt plus 1pt minus 1pt}  
   {{\thmnumber{#1} \thmname{#2}{\thmnote{ (#3)}}}}          
\newtheorem{innercustomthm}{Corollary}

\newtheorem{othercustomthm}{Theorem}

\theoremstyle{plain2}

\newtheorem{thmy}{Theorem}
\renewcommand{\thethmy}{\Alph{thmy}} 
\newenvironment{thmx}{\stepcounter{Theorem}\begin{thmy}}{\end{thmy}}
\newtheorem{theorem}[equation]{Theorem}
\newtheorem{corollary}[equation]{Corollary}
\newtheorem{lemma}[equation]{Lemma}
\newtheorem{proposition}[equation]{Proposition}

\theoremstyle{definition}
\newtheorem{definition}[equation]{Definition}
\newtheorem{notation}[equation]{Notation}

\newtheorem{claim}[equation]{Claim}

\newtheorem{remark}[equation]{Remark}
\newtheorem{example}[equation]{Example}

\newtheoremstyle{stepstyle}
   {}     {}
   {\normalfont}
   {\parindent}
   {\itshape}
   {}
   {5pt plus 1pt minus 1pt}
   {{\thmname{#1} \thmnumber{#2}:{\thmnote{#3}}}}
\theoremstyle{stepstyle}

\newtheoremstyle{point}
   {}     {}
   {\normalfont}
   {}
   {\bfseries}
   {}
   {5pt plus 1pt minus 1pt}
   {{\thmname{#1}(\thmnumber{#2})\thmnote{ #3.}}}
\theoremstyle{point}
\newtheorem{point}[equation]{}
\newcommand{\pa}[1]{\begin{point}#1\end{point}}              
\newcommand{\Pa}[2]{\begin{point}[#1]#2\end{point}}         
\newtheoremstyle{subpoint}
   {}     {}           
   {\normalfont}
   {}                 
   {\normalfont}
   {}
   {5pt plus 1pt minus 1pt}
   {{\thmname{#1}{\bf (\thmnumber{#2})}\thmnote{ #3.}}}
\theoremstyle{subpoint}
\newtheorem{subpoint}[equation]{}
\newcommand{\spa}[1]{\begin{subpoint}#1\end{subpoint}}           
\newcommand{\Spa}[2]{\begin{subpoint}[#1]#2\end{subpoint}}   
\usepackage{combelow}
\usepackage{tikz-cd}
\usepackage{enumitem}

\crefformat{equation}{(#2#1#3)}

\makeatletter
\def\@tocline#1#2#3#4#5#6#7{\relax
  \ifnum #1>\c@tocdepth 
  \else
    \par \addpenalty\@secpenalty\addvspace{#2}%
    \begingroup \hyphenpenalty\@M
    \@ifempty{#4}{%
      \@tempdima\csname r@tocindent\number#1\endcsname\relax
    }{%
      \@tempdima#4\relax
    }%
    \parindent\z@ \leftskip#3\relax \advance\leftskip\@tempdima\relax
    \rightskip\@pnumwidth plus4em \parfillskip-\@pnumwidth
    #5\leavevmode\hskip-\@tempdima
      \ifcase #1
       \or\or \hskip 1em \or \hskip 2em \else \hskip 3em \fi%
      #6\nobreak\relax
    \dotfill\hbox to\@pnumwidth{\@tocpagenum{#7}}\par
    \nobreak
    \endgroup
  \fi}

\renewcommand{\O}{\mathcal{O}}

\newcommand{\X}{\mathfrak{X}}

\newcommand{\Frac}{\textrm{Frac}}

\newcommand{\pr}{\textrm{pr}}

\newcommand{\M}{\mathcal{M}}

\renewcommand{\P}{\mathbf{P}}

\newcommand{\del}{\partial}
\newcommand{\an}{\textrm{an}}

    \usepackage{soul}
    
\usepackage{mathtools}

\newcommand{\NNN}{\mathbf{N}}

\renewcommand{\H}{\mathcal{H}}

\DeclareMathOperator{\An}{An}

\DeclareMathOperator{\triv}{triv}
\DeclareMathOperator{\hyb}{hyb}
\DeclareMathOperator{\DD}{D}
\DeclareMathOperator{\BB}{B}
\DeclareMathOperator{\ev}{ev}
\newcommand{\Q}{\mathbf{Q}}
\newcommand{\R}{\mathbf{R}}
\newcommand{\Z}{\mathbf{Z}}
\newcommand{\A}{\mathbf{A}}
\newcommand{\C}{\mathbf{C}}
\DeclareMathOperator{\Spec}{Spec}
\newcommand{\D}{\mathbf{D}}
\newcommand{\E}{\mathbf{E}}
\newcommand{\frakp}{\mathfrak{p}}

\newcommand{\frakr}{\mathfrak{r}}
\DeclareMathOperator{\ord}{ord}
\DeclareMathOperator{\Hom}{Hom}
\DeclareMathOperator{\arch}{arch}

\renewcommand{\lnot}{\mathord{\sim}}

\author{Thibaud Lemanissier and Matthew Stevenson}
\title{Topology of Hybrid Analytifications}
\date{\today}
\address{Department of Mathematics, University of Michigan, Ann Arbor, MI
48109--1043, USA}
\email{\href{mailto:stevmatt@umich.edu}{stevmatt@umich.edu}}
\begin{document}

\maketitle


\begin{abstract}
We investigate the topological properties of Berkovich analytifications over hybrid fields, that is a field equipped with the maximum of its native norm and the trivial norm. 
We prove that the analytification of the affine line or of a smooth projective curve over a countable Archimedean hybrid field is contractible, and show that it can be non-contractible when the field is uncountable.
Further, we prove that the analytification of affine space over a non-Archimedean hybrid field or over a discrete valuation ring is contractible.
As an application, we show that the Berkovich affine line over the ring of integers of a number field is contractible. 
\end{abstract}

\tableofcontents

\section{Introduction}

Let $k$ be a field equipped with a nontrivial absolute value $| \cdot |_k$, which
may be Archimedean or non-Archimedean, and is not assumed to be complete.
The field $k$ can also be thought of as a Banach ring equipped with the hybrid norm
$$
\| \cdot \|_{\hyb} \coloneqq \max\{ | \cdot |_k, | \cdot |_0 \},
$$
where $| \cdot |_0$ denotes the trivial norm on $k$.
To a variety $X$ over $k$, there is a \emph{hybrid analytification} 
$$
\lambda \colon X^{\hyb} \to \M(k, \| \cdot \|_{\hyb}) \simeq [0,1],
$$
which is a locally compact, Hausdorff topological space equipped with natural identifications 
$$
\begin{cases}
X^{\triv} = \lambda^{-1}(0),\\
X^{\an} = \lambda^{-1}(1),
\end{cases}
$$
where $X^{\an}$ and $X^{\triv}$ denote the analytifications of $X$ with respect to $(\widehat{k},| \cdot |_k)$ and $(k,| \cdot |_0)$, respectively (in the sense of~\cite{berkovich}).
Said differently, $X^{\hyb}$ is a family of analytic spaces that interpolates between $X^{\triv}$ and $X^{\an}$. 
When $k$ is the complex numbers equipped the Archimedean norm, $X^{\hyb}$ is a space that consists of both Archimedean and non-Archimedean data, hence the name `hybrid' analytification.

These hybrid spaces were first introduced in~\cite{berkovich-hodge} to give a non-Archimedean interpretation of the weight-zero piece of the mixed Hodge structure on a proper variety over $\mathbf{C}$. 
The properties of hybrid spaces (and, more generally, of analytifications over Banach rings) have been explored further in~\cite{poineau1,lemanissier}, and they have found additional applications in~\cite{mattias-hybrid,favre16,boucksom-jonsson,krieger}.

The goal of this paper is to investigate connections between the topology of $X^{\hyb}$ and that of $X^{\triv}$ and $X^{\an}$. More precisely, we pose the following question, inspired by~\cite[Remark 2.5(ii)]{berkovich-hodge}.

\begin{question}\label{main question}
Is the inclusion $X^{\triv} \hookrightarrow X^{\hyb}$ a homotopy equivalence?
\end{question}

Given a positive answer to the question, one can define a natural ``specialization'' map 
$$
H^{\ast}(X^{\triv},\Q) \to H^{\ast}(X^{\an},\Q)
$$
on (e.g.) singular cohomology, following~\cite[\S 2]{berkovich-hodge}. The topology of $X^{\triv}$ is some measure of the singularities of $X$, so this specialization map would offer insight as to how the singularities of $X$ inform the topology of $X^{\an}$.


However, the answer to the question is no in general: 
if $(k, | \cdot |_k)$ is Archimedean and $X = \A^1_k$ is the affine line over $k$, then $X^{\mathrm{triv}}$ is always contractible but $X^{\hyb}$ need not be. This is demonstrated by the first main result.
\begin{thmx}\label{theorem 1}
Let $k$ be a subfield of $\C$ equipped with the Archimedean norm $| \cdot |_{\infty}$.
\begin{enumerate}
\item If $k$ is countable, then there is a strong deformation retraction of $\A^{1,\hyb}_k$ onto $\A^{1,\triv}_k$; in particular, $\A^{1,\hyb}_k$ is contractible.
\item If $k$ is uncountable and not included in $\mathbf{R}$, then $\A^{1,\hyb}_k$ is not contractible.
\end{enumerate}
\end{thmx}

The idea for one direction of the proof of~\cref{theorem 1} is simple to state: if $k$ is uncountable, there is a point of $\lambda^{-1}(0)$ with no countable basis of open neighbourhoods, whereas $\lambda^{-1}((0,1])$ is metrizable. 
A careful study of the local structure of $\A^{1,\hyb}_k$ near this point of $\lambda^{-1}(0)$ ultimately leads to the proof of  non-contractibility.

When $k$ is countable, the proof of~\cref{theorem 1} shows more generally that any Zariski-open subset of $\P^{1,\hyb}_k$ is contractible. This can be pushed further to give the second main result.

\begin{thmx}\label{theorem 12}
Let $k$ be a countable subfield of $\C$, equipped with the Archimedean norm $| \cdot |_{\infty}$.
If $X$ is a smooth projective curve over $k$, then there is a strong deformation retraction of $X^{\hyb}$ onto $X^{\triv}$; in particular, $X^{\hyb}$ is contractible. 
\end{thmx}

The results of~\cite{locally_contractible_2} show that the trivially-valued analytification of a smooth variety is contractible; in particular,~\cref{theorem 12} gives a positive answer to the question in this special case.

Now, if one demands that $(k,|\cdot |_k)$ be a complete non-Archimedean field, there appears to be more hope for the question to yield a positive answer. This is exemplified by the third main result.

\begin{thmx}\label{theorem 2}
If $(k,|\cdot|_k)$ is a complete non-Archimedean field, then $\A^{n,\triv}_k \hookrightarrow \A^{n,\hyb}_k$ is a homotopy equivalence. 
\end{thmx}

The proof of~\cref{theorem 2} is based on the constructions of~\cite{berkovich,thuillier}: one uses the torus action on $\A^n_k$ to construct a ``hybrid toric skeleton'' in $\A^{n,\hyb}_k$
(this coincides with the image of the canonical section of the hybrid tropicalization map of~\cite{mattias-hybrid}).
More generally, the authors expect~\cref{theorem 2} to hold for any normal toric variety. 
See~\cite[\S 7]{abramovich} for a nice exposition of this construction in the classical setting.
In fact, the same method of proof shows the analogous result for affine spaces over 
a discrete valuation ring (dvr) or over a Dedekind domain, equipped with the trivial norm.




The methods developed in the proofs of the above results facilitate a description of the topology of the affine line over $(\Z,|\cdot |_{\infty})$ and, more generally, over the ring of integers of a number field.
The local structure of analytifications of schemes over such Banach rings has been explored in~\cite{poineau1}, but their global topological structure is not well understood. 
To that end, we offer an application of the above results.

\begin{thmx}\label{theorem 3}
The affine line over the ring of integers of a number field is contractible.
\end{thmx}

The proof of~\cref{theorem 3} follows by combining the construction in~\cref{theorem 1} with the aforementioned generalization of~\cref{theorem 2} to the affine line over a Dedekind domain equipped with the trivial norm.

\subsection{Organization}
In \S 2, we recall the construction of the analytification of a scheme over a Banach ring, and prove a lifting theorem for strong deformation retractions along finite branched covers. 
In~\S\ref{local structure}, we develop the preliminary results on the local structure of hybrid analytifications that are needed to prove~\cref{theorem 1} and~\cref{theorem 12}, whose proofs follow in \S 3.2-3.4.
The proof of~\cref{theorem 2} spans \S 4, and the proof of~\cref{theorem 3} follows in \S 5.


\subsection{Acknowledgements}
The authors would like to thank Antoine Ducros and Jérôme Poineau for putting the authors in touch and for helpful comments on the contents of the paper, as well as Tyler Foster, Jifeng Shen, and Martin Ulirsch for interesting discussions related to the paper. 
The second author would like to thank his advisor, Mattias Jonsson, for many helpful comments on an earlier version, which greatly improved the presentation.
The second author was partially supported by NSF grant DMS-1600011. 


\section{Preliminaries}




\subsection{Analytifications over Banach rings}\label{topological analytifications}
\spa{
A \emph{Banach ring} is a pair $(A, \| \cdot \|)$, where $A$ is a commutative noetherian ring, and $\| \cdot \| \colon A \to \R_{\geq 0}$ is a submultiplicative norm with respect to which $A$ is complete; when the norm $\| \cdot \|$ is implicit, we simply write $A$ for the pair $(A,\| \cdot \|)$. 

The \emph{spectrum} $ \M(A,\| \cdot \|) $ of $(A,\| \cdot \|)$ is the set of multiplicative seminorms $| \cdot | \colon A \to \R_{\geq 0}$ such that $| \cdot | \leq \| \cdot \|$, equipped with the topology of pointwise convergence.
This space is compact and Hausdorff, and it is non-empty when $A$ is nonzero. 
The spectrum comes equipped with a continuous map $\rho_A \colon \M(A, \| \cdot \|) \to \Spec(A)$ that sends a seminorm to its kernel. 
More generally, the spectrum can be defined for normed rings that are not necessarily complete, but the construction factors through the separated completion.
See~\cite[\S 1]{berkovich} for more details.


Some classic examples of Banach rings include a field that is complete with respect to a norm, any ring equipped with the trivial norm, or the integers equipped with the Archimedean absolute value. See~\cite[\S 1.4]{berkovich} for a discussion of these Banach rings and their spectra.
}

\begin{example}\label{hybrid point}
Let $k$ be a field equipped with a non-trivial norm $| \cdot |_k$, and let $| \cdot |_0$ denote the trivial norm on $k$. The \emph{hybrid norm} $\| \cdot \|_{\mathrm{hyb}} \coloneqq \max\{ |\cdot |_k, | \cdot |_0 \}$ is a submultiplicative norm on $k$, and $(k, \| \cdot \|_{\mathrm{hyb}})$ is complete (regardless of the completeness of $(k,|\cdot |_k)$!).
There is a continuous injection $[0,1] \hookrightarrow \M(k,\| \cdot \|_{\mathrm{hyb}})$ given by
$$
\rho \mapsto | \cdot |_k^{\rho},
$$
where $| \cdot |_k^{0}$ is interpreted as the trivial norm $| \cdot |_0$. 
Arguing as in~\cite[\S 2]{berkovich-hodge}, one can check that this is in fact a homeomorphism.
The space $\M(k,\| \cdot \|_{\mathrm{hyb}})$ is 
called
the \emph{hybrid point} over $(k, |\cdot |_k)$, and
the pair 
$(k, \| \cdot \|_{\mathrm{hyb}})$
is the \emph{hybrid field}.
\end{example}


\spa{
Fix a Banach ring $(A, \| \cdot \|)$. 
For a finite type $A$-scheme $X$, Berkovich introduced in~\cite[\S 1]{berkovich-hodge} a notion of analytification $X^{\An} \to \M(A)$ of $X$ over the Banach ring $A$.
More precisely,
there is a locally compact topological space $X^{\An}$, called the \emph{analytification} of $X$ over $A$, which is equipped with a continuous map $\lambda = \lambda_X \colon X^{\An} \to \M(A)$. 
The construction of $X^{\An}$ proceeds in two steps:
\begin{enumerate}[label=(\alph*)]
\item If $X = \Spec(B)$ is affine, then $X^{\An}$ is the set of multiplicative seminorms $| \cdot |\colon B \to \R_{\geq 0}$ such that $| a | \leq \| a \|$ for all $a \in A$;
$X^{\An}$ is equipped with the weakest topology such that $ | \cdot | \mapsto |f |$ is continuous for each $f \in B$. 
For a seminorm $x \in X^{\An}$ and $f \in B$, write $|f(x)|$ for the value of $x$ on $f$. 
The map $\lambda \colon X^{\An} \to \M(A)$ is given by restriction of the seminorms to $A$.
\item In general, $X$ can be covered by finitely-many open affine subsets and the previous procedure glues to give the space $X^{\An}$ and the map $\lambda \colon X^{\An} \to \M(A)$. 
\end{enumerate}
}
\noindent Furthermore, there is continuous map $\rho_X \colon X^{\An} \to X$, called the \emph{kernel map}, that fits into a commutative diagram of continuous maps
\begin{center}
\begin{tikzcd}
X^{\An} \arrow{r}{\rho_X} \arrow[swap]{d}{\lambda} & X \arrow{d}{} \\
\M(A) \arrow{r}{\rho_A} & \Spec(A).
\end{tikzcd}
\end{center}
When $X = \Spec(B)$ is affine, $\rho_X$ is the map that sends a seminorm on $B$ to its kernel (hence the name).

\spa{
For $x \in X^{\An}$, write $\H(x) = \H_X(x)$ for the \emph{completed residue field} of $X^{\An}$ at $x$. If $X = \Spec(B)$ is affine, then $\H(x)$ is the completion of $\mathrm{Frac}(B/\ker(x))$ with respect to the residue norm induced by $x$. In general, $\H(x)$ can be computed in this manner after passing to an affine open containing $\ker(x)$.
}

\spa{
Working affine-locally, it is easy to check that the assignment $X \mapsto X^{\An}$ is functorial; for a morphism $\varphi \colon Y \to X$ between finite-type $A$-schemes, write $\varphi^{\An} \colon Y^{\An} \to X^{\An}$ for the corresponding continuous map between analytifications. 
In addition, the formation of the kernel map is functorial; that is, $\varphi \circ \rho_Y = \rho_X \circ \varphi^{\An}$.
}


\spa{
The analytification functor satisfies certain (topological) GAGA theorems; that is, there are properties of a morphism of schemes that translate to topological properties of the induced map on analytifications.
The relevant topological conditions are recalled below:
for a continuous map $\phi \colon V \to U$ of topological spaces, we say that
\begin{enumerate}[label=(\alph*)]
\item $\phi$ is \emph{Hausdorff} (or \emph{separated}) if the diagonal $V \to V \times_U V$ is a closed map, and this is always satisfied if $V$ is Hausdorff;
\item $\phi$ is \emph{proper} if the preimage of any (quasi-)compact set is (quasi-)compact, and this is equivalent to $\phi$ being a closed map with (quasi-)compact fibres when $U$ is locally compact and Hausdorff;
\item $\phi$ is \emph{finite} if it is a closed map with finite fibres.
\end{enumerate}
\noindent The (topological) GAGA results needed in the sequel are recorded in the following proposition.

\begin{proposition}\label{topological gaga}
Let $\varphi \colon Y \to X$ be a morphism between finite-type $A$-schemes.
\begin{enumerate}[label=(\roman*)]
\item If $\varphi$ is an open (resp.\ closed) immersion, then $\varphi^{\An}$ is a homeomorphism onto an open (resp.\ closed) subspace.
\item If $\varphi$ is separated, then $\varphi^{\An}$ is Hausdorff.
\item If $\varphi$ is proper, then $\varphi^{\An}$ is proper.
\item If $\varphi$ is finite, then $\varphi^{\An}$ is finite.
\end{enumerate}
\end{proposition}

\begin{proof}
The assertion (i) is~\cite[Lemma 1.1(ii)]{berkovich-hodge}, (ii) and (iii) are~\cite[Lemma 1.2]{berkovich-hodge}, and (iv) is~\cite[Proposition 6.24]{lemanissier}.
\end{proof}
}


\spa{
The analytification over a Banach ring captures many notions that exist already in the literature. 
Three such examples are listed below.
\begin{enumerate}[label=(\alph*)]
\item If $(A,\| \cdot \|) = (\mathbf{C},| \cdot |_{\infty})$, then $X^{\An}$ is the usual complex analytification $X^{h} = X(\mathbf{C})$ of $X$.
\item If $(A,\| \cdot \|) = (\mathbf{R}, | \cdot |_{\infty})$, then $X^{\An}$ is an $\R$-analytic space in the sense of~\cite[\S 1]{berkovich-vanishing-cycles}, whose underlying space is the quotient of $X(\C)$ by complex conjugation. 
This is not to be confused with the classical notion of a real-analytic space, which consists of the points of $X(\C)$ that are fixed by complex conjugation.
\item If $(A,\| \cdot \|) = (k, |\cdot |_k)$ is a complete non-Archimedean field, then $X^{\An}$ is the analytification $X^{\an}$ of $X$ in the sense of~\cite{berkovich,berkovich93}.
\end{enumerate}
While additional exotic examples appear elsewhere (see e.g.\ \cite{berkovich93,poineau1,mattias-hybrid,boucksom-jonsson}), the example at the heart of this paper is the case when $A$ is a hybrid field, which is discussed in the example below.
}

\begin{example}\label{hybrid analytification}
Let $(k, \| \cdot \|_{\mathrm{hyb}})$ be a hybrid field as in~\cref{hybrid point}.
For a finite-type $k$-scheme $X$, the analytification $\lambda \colon X^{\hyb} \to \M(k , \| \cdot \|_{\mathrm{hyb}}) \simeq [0,1]$ is called the \emph{hybrid analytification} of $X$. 
As detailed in~\cite[\S 2]{berkovich-hodge}, there are identifications 
$$
\begin{cases}\lambda^{-1}(0) = X^{\mathrm{triv}},\\
\lambda^{-1}(1) = X^{\an},
\end{cases}
$$
where $X^{\mathrm{triv}}$ and $X^{\mathrm{an}}$ denote the analytifications (in the sense of~\cite{berkovich}) with respect to $(k, | \cdot |_0)$ and $(\widehat{k},| \cdot |_k)$, respectively. 
Moreover, there is a homeomorphism $\lambda^{-1}((0,1]) \simeq (0,1] \times X^{\an}$ over $(0,1]$; when $X = \Spec(B)$ is affine, a seminorm $x \in X^{\hyb}$ is sent to $\lambda(x) \in (0,1]$ and the seminorm $f \mapsto |f(x)|^{1/\lambda(x)}$ on $B$. 
The formation of this homeomorphism $\lambda^{-1}((0,1]) \simeq (0,1] \times X^{\an}$ is easily seen to be functorial in $X$. 

If $x \in \lambda^{-1}((0,1])$ corresponds to $(\rho,x') \in (0,1] \times X^{\an}$, then the completed residue fields $\H_{X^{\hyb}}(x)$ and $\H_{X^{\an}}(x')$ are canonically isometric. Indeed, if $X = \Spec(B)$ is affine, then $\ker(x) = \ker(x')$ and so both fields are completions of the field $\Frac(B/\ker(x))$ with respect to equivalent norms.

Further, if $k$ is a subfield of $\C$ and $|\cdot |_k = |\cdot |_{\infty}$ is the Archimedean norm,
then the fibre $\lambda^{-1}(\rho)$, for $\rho \in (0,1]$, is a complex-analytic space if $k$ is not included in $\R$, and it is an $\R$-analytic space otherwise. 
On the other hand, the fibre $\lambda^{-1}(0)$ is a non-Archimedean analytic space over $(k, |\cdot |_0)$. 
For this reason, we say that the $\lambda^{-1}(\rho)$'s, for $\rho \in (0,1]$, are the \emph{Archimedean fibres} of $X^{\hyb}$, and $\lambda^{-1}(0)$ is the \emph{non-Archimedean fibre}. 
%
%
%
\end{example}

\begin{proposition}\label{prop:path-connected}
Let $(k,\| \cdot \|_{\hyb})$ be a hybrid field, and let $X$ be a finite-type $k$-scheme.
Suppose that one of the following two conditions hold:
\begin{enumerate}[label=(\roman*)]
\item $X(k) \not= \emptyset$;
\item $k$ is complete with respect to $| \cdot |_k$.
\end{enumerate}
If $X$ is connected, then $X^{\hyb}$ is path-connected.
\end{proposition}

In addition, \cite[Th\'eor\`eme 6.14]{lemanissier} shows the deeper result that a hybrid analytification is locally path-connected (while the proof is written in the case where the base ring is the ring of integers of a number field, it goes through verbatim in the case of a hybrid field).
This result is the hybrid analogue of~\cite[Theorem 3.2.1]{berkovich}. 

\begin{proof}
The strategy of the proof follows~\cite[Corollary 2.2]{berkovich-hodge}.
Assume first that $X$ has a $k$-point $z \in X(k)$.
The point $z$ induces a continuous section $\sigma_z$ of the structure map $\lambda \colon X^{\hyb} \to [0,1]$, given Zariski-locally by 
$$
|f(\sigma_z(\rho))| \coloneqq |f \bmod \mathfrak{m}_z |_k^{\rho},
$$
where $U = \Spec(B)$ is an affine-open neighbourhood of $z$ of $X$, $\mathfrak{m}_z \subseteq B$ is the maximal ideal corresponding to $z$, $f \in B$, and $\rho \in [0,1]$. 
Each fibre of $\lambda$ is path-connected by~\cite[XII, Proposition 2.4]{sga1} and~\cite[Theorem 3.4.8(ii)]{berkovich}, so it suffices to produce a path between points $x$ and $x'$ that lie in two distinct fibres of $\lambda$. 
This can be done as follows: take a path in $\lambda^{-1}(\lambda(x))$ from $x$ to $\sigma_z(\lambda(x))$, compose it with the path $t \mapsto \sigma_z(t)$ for $t \in [\lambda(x),\lambda(x')]$, and then with a path in $\lambda^{-1}(\lambda(x'))$ from $\sigma_z(\lambda(x'))$ to $x'$. 

Assume now that $k$ is complete and $X$
does not admit a $k$-point.
Pick a finite extension $K/k$ such that $X$ admits a $K$-point, and the completeness of $k$ guarantees that there is a unique extension $|\cdot |_K$ of the norm $|\cdot |_k$ on $k$ to one on $K$. 
By the previous case, the hybrid analytification $X_K^{\hyb}$ of $X_K = X \times_k K$ with respect to the hybrid norm on $K$ is path-connected. 
By~\cite[\S 4.3]{lemanissier}, there is a continuous, surjective ground field extension map $X_K^{\hyb} \to X^{\hyb}$; in particular, $X^{\hyb}$ is also path-connected.
\end{proof}

\begin{remark}
The proof of~\cref{prop:path-connected} shows the following more general assertion.
Suppose $A$ is a Banach ring with $\M(A)$ path-connected and $X$ is a connected finite-type $A$-scheme such that there exists a Banach $A$-algebra $A'$ with $X(A') \not= \emptyset$, $\M(A')$ path-connected, and $\M(A') \to \M(A)$ surjective. Then, the analytification $X^{\An}$ is path-connected.
\end{remark}

\begin{example}\label{hybrid affine line example}
Let $k$ be a subfield of $\C$ equipped with the Archimedean norm $|\cdot |_{\infty}$. 
The hybrid analytification $\A^{1,\hyb}_k$ of the affine line is a central example in this paper, and its points can be described explicitly.

The Archimedean points of $\A^{1,\hyb}_k$ will be described first: given $z \in \C$ and $t \in (0,1]$, write $\ev(z,t)$ for the seminorm in $\lambda^{-1}((0,1])$ given by
$$
f \mapsto |f(\ev(z,t))| \coloneqq |f(z)|_{\infty}^t.
$$
Following~\cref{hybrid analytification}, the homeomorphism-type of $\lambda^{-1}((0,1])$ is easy to describe:
\begin{enumerate}[label=(\alph*)]
\item If $k$ is not contained in $\mathbf{R}$, then the map $(z,t) \mapsto \ev(z,t)$ is a homeomorphism $(0,1] \times \C \simeq \lambda^{-1}((0,1])$.
\item If $k$ is contained in $\mathbf{R}$, then $\ev(z,t) = \ev(\overline{z},t)$, and $(z,t) \mapsto \ev(z,t)$ descends to a homeomorphism $(0,1] \times (\C/\lnot) \simeq \lambda^{-1}((0,1])$, where $\sim$ denotes the equivalence relation on $\mathbf{C}$ generated by complex conjugation.
\end{enumerate}
\noindent While the topological space $\lambda^{-1}((0,1])$ depends only on whether or not $k$ is a subfield of $\R$, the topology of $\A^{1,\hyb}_k$ and of $\lambda^{-1}(0)$ very much do depend on the field $k$. 

The fibre $\lambda^{-1}(0)$ is identified with the trivially-valued analytification $\A^{1,\mathrm{triv}}_k$, whose points are well-understood.
To each monic, irreducible polynomial $p \in k[T]$ and $r \in \R_{\geq 0}$, we associate a point $\eta_{p,r} \in \lambda^{-1}(0)$ as follows: 
for any root $z \in \C$ of $p$, 
define a multiplicative seminorm $k[T] \to \R_{\geq 0}$ by the formula
$$
f = \sum_{i \geq 0}a_i (T-z)^i \mapsto |f(\eta_{p,r})| \coloneqq \max_{i \geq 0} |a_i|_0 r^i.
$$
The seminorm $\eta_{p,r}$ is independent of the choice of root $z$ of $p$.
If $p = T-z$ is linear, write $\eta_{z,r} \coloneqq \eta_{p,r}$.
As explained in~\cite[\S 1.4.4]{berkovich}, every point of $\A^{1,\triv}_k$ arises in this manner. 

The map $[0,1] \to \A^{1,\hyb}_k$, given by $r \mapsto \eta_{p,r}$, is an injective continuous map with image in $\lambda^{-1}(0)$, and we denote its image by $[\eta_{p,0},\eta_{p,1}]$. 
The open interval $[\eta_{p,0},\eta_{p,1})$ is called a \emph{branch} of $\A^{1,\hyb}_k$.
As we range over all such polynomials $p$, the image of the intervals $[\eta_{p,0},\eta_{p,1}]$ sweeps out a compact subspace of $\A^{1,\triv}_k$.
The points not lying in any branch are those of the form $\eta_{0,r}$ for $r \geq 1$ (note that for any polynomials $p$ and $q$, $\eta_{p,r} = \eta_{q,r}$ whenever $r \geq 1$).

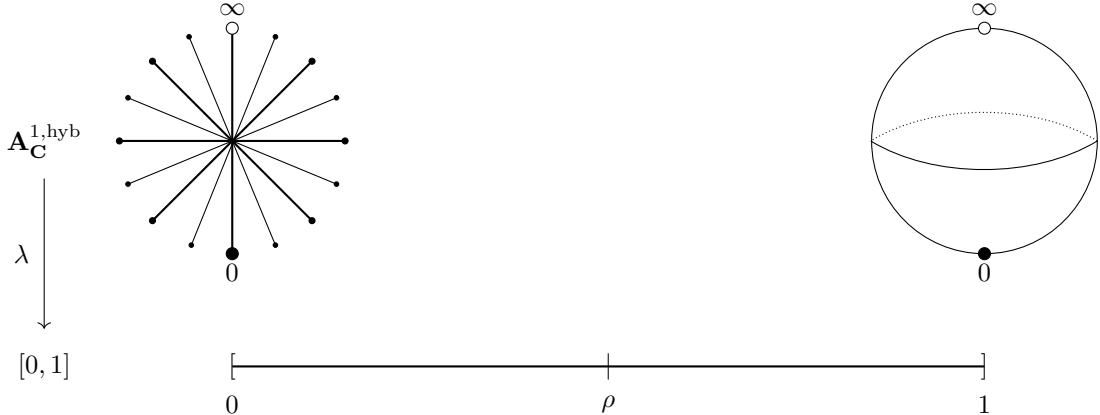
\begin{figure}[h]
  \centering
  \begin{tikzpicture}

    \draw[thick] (-7,-3) -- (3,-3);
    \node at (-7,-3) {$[$};
    \node at (3,-3) {$]$};
    \node at (3,-3.5) {$1$};
    \node at (-7,-3.5) {$0$};
    \node at (-2,-3.5) {$\rho$};
    \node at (-2,-3) {$|$};
    \draw[->] (-9.5,-0.5) -- (-9.5,-2.5);
    \node at (-9.8,-1.5) {$\lambda$};
    \node at (-9.5,0) {$\A^{1,\hyb}_{\C}$};
    \node at (-9.5,-3) {$[0,1]$};
    \draw (3,0) circle (1.5cm);
    \draw[fill=white] (3,1.5) circle (0.08cm);
    \node at (3,1.75) {$\infty$};
    \draw[fill=black] (3,-1.5) circle (0.08cm);
    \node at (3,-1.75) {$0$};
    \draw (4.5,0) arc (315:225:2.13cm and 1.3cm);
    \draw[densely dotted] (4.5,0) arc (45:135:2.13cm and 1.3cm);
    
    \draw[thick] (-7,-1.5) -- (-7,1.5);
    \draw[fill=white] (-7,1.5) circle (0.08cm);
    \node at (-7,1.75) {$\infty$};
    \draw[fill=black] (-7,-1.5) circle (0.08cm);
    \node at (-7,-1.75) {$0$};
    \draw[thick] (-8.5,0) -- (-5.5,0);
    \draw[fill=black] (-8.5,0) circle (0.04cm);
    \draw[fill=black] (-5.5,0) circle (0.04cm);
    \draw[thick] (-8.06066,0.43934-1.5) -- (-5.93934,2.56066-1.5);
    \draw[fill=black] (-8.06066,0.43934-1.5) circle (0.04cm);
    \draw[fill=black] (-5.93934,2.56066-1.5) circle (0.04cm);
    \draw[thick] (-8.06066,1.06066) -- (-5.93934,-1.06066);
    \draw[fill=black] (-8.06066,1.06066) circle (0.04cm);
    \draw[fill=black] (-5.93934,-1.06066) circle (0.04cm);
    \draw (-5.61418,0.574025) -- (-8.38582,-0.574025);
    \draw[fill=black] (-5.61418,0.574025) circle (0.03cm);
    \draw[fill=black] (-8.38582,-0.574025) circle (0.03cm);
    \draw (-6.42597,1.38582) -- (-7.54403,-1.38582);
    \draw[fill=black] (-6.42597,1.38582) circle (0.03cm);
    \draw[fill=black] (-7.54403,-1.38582) circle (0.03cm);
    \draw (-7.57403,1.38582) -- (-6.42597,-1.38582);
    \draw[fill=black] (-7.57403,1.38582) circle (0.03cm);
    \draw[fill=black] (-6.42597,-1.38582) circle (0.03cm);
    \draw (-8.38582,0.574025) -- (-5.61418,-0.574025);
    \draw[fill=black] (-8.38582,0.574025) circle (0.03cm);
    \draw[fill=black] (-5.61418,-0.574025) circle (0.03cm);
    
  \end{tikzpicture}
  \caption{The hybrid affine line $\A^{1,\hyb}_{\C}$ over $\C$ interpolates between the complex plane (pictured on the right) and the trivially-valued Berkovich analytification of $\A^1_{\C}$ (pictured on the left).}
  \label{fig:hybrid_line}
\end{figure}  

The point $\eta_{0,1}$ plays a special role: if $k$ is uncountable, $\eta_{0,1}$ is the only point of $\A^{1,\hyb}_k$ without a countable basis of open neighbourhoods.
In fact, any open neighbourhood of $\eta_{0,1}$ contains cofinitely-many branches of $\A^{1,\hyb}_k$.
\end{example}

\subsection{Analytifications over Geometric Base Rings}

\spa{There is a class of Banach rings, known as \emph{geometric base rings}, for which the properties of the analytification functor are much better understood; in particular, there is a reasonable category of analytic spaces over such Banach rings. 
These are rings of dimension at most 1 
whose spectra satisfy certain topological and analytic conditions, and they are
defined and studied in
~\cite[D\'efinition 3.22]{lemanissier} (building on the earlier work of~\cite[D\'efinition 8.5]{poineau3}).
The definition of a geometric base ring is quite technical, and it will not be repeated here. 
Instead, we record a large class of examples below (that includes most Banach rings considered in this paper).
}

\begin{example}\label{geometric base ring examples}
The main examples of geometric base rings are listed below:
\begin{enumerate}[label=(\roman*)]
    \item a field that is complete with respect to an absolute value;
    \item a hybrid field, as in~\cref{hybrid point};
    \item a dvr, equipped with the trivial norm; 
    \item the ring of integers of a number field, equipped with the trivial norm;
    \item the ring of integers of a number field, equipped with the maximum of the Archimedean norms induced by its complex embeddings.
\end{enumerate}
See~\cite[Remarque 8.6]{poineau3} and~\cite[Remarque 1.26]{lemanissier} for further examples and details.
\end{example}

\spa{
For a geometric base ring $A$, there is a category $\mathrm{An}_A$ of $A$-analytic spaces, which was constructed in~\cite[\S 2]{lemanissier}; it can be realized as a (not necessarily full) subcategory of the category $\mathrm{LRS}_A$ of locally $A$-ringed spaces.
In particular, any object $Z$ in $\mathrm{An}_A$ is equipped with a structure sheaf $\O_Z$ of analytic functions.
For example, if $A = (\C,|\cdot |_{\infty})$, then $\mathrm{An}_A$ is the category of complex-analytic spaces; if $A = k$ is a complete non-Archimedean field, then $\mathrm{An}_A$ is Berkovich's category of $k$-analytic spaces.
}

\spa{A key result of~\cite{lemanissier} is that the analytification of a scheme over $A$ (in the sense of \S\ref{topological analytifications}) is naturally equipped with the structure of an $A$-analytic space. 
More precisely, the following analogue of~\cite[Theorem 3.4.1]{berkovich} holds in this setting:

\begin{theorem}\cite[Th\'eor\`eme 4.4]{lemanissier}\label{analytifications}
Let $A$ be a geometric base ring and let $X$ be a finite-type $A$-scheme. 
The functor $\mathrm{An}_A \to (\mathrm{Sets})$, given by $$
\mathfrak{X} \mapsto \Hom_{\mathrm{LRS}_A}(\mathfrak{X},X),
$$
is representable by an $A$-analytic space $X^{\An}$ and a morphism $\rho_X \colon X^{\An} \to X$ of locally $A$-ringed spaces.
\end{theorem}

\noindent It follows from the proof of~\cref{analytifications} that the topological space underlying $X^{\An}$ coincides with the analytification of $X$ from \S\ref{topological analytifications}, and the continuous map underlying $\rho_X$ is the kernel map. 
For this reason, we do not make any distinction in the notation.
}

\spa{
The analytification functor over a geometric base ring satisfies certain GAGA properties (see~\cite[\S 6.2]{lemanissier}).  
One such result is given below and we develop other partial results throughout the section in the case when the base is a hybrid field.
}

\begin{definition}
A morphism $\varphi \colon Y \to X$ of $A$-analytic spaces is \emph{flat} (resp.\ \emph{unramified}, \emph{\'etale}) if for every $y \in Y$, the local ring homomorphism $\O_{X,\varphi(y)} \to \O_{Y,y}$ is flat (resp.\ unramified, \'etale).
\end{definition}

\begin{theorem}\cite[Proposition 6.28]{lemanissier}\label{finite flat gaga}
Let $A$ be a geometric base ring in~\cref{geometric base ring examples}, and let $\varphi \colon Y \to X$ be a morphism between finite-type $A$-schemes.
If $\varphi$ is finite and flat, then $\varphi^{\An}$ is finite and flat. 
\end{theorem}

\spa{
We expect~\cref{finite flat gaga} to also hold for finite unramified morphisms, and hence for finite \'etale morphisms as well.
To do so, one would need to show that a (finite) morphism $\varphi \colon Y \to X$ of $A$-analytic spaces is unramified if and only if for all $x \in X$, the fibre $\varphi^{-1}(x)$ is a disjoint union of spectra of finite separable $\H(x)$-algebras. 
This requires developing a theory of modules of K\"ahler differentials as in~\cite[\S 3.3]{berkovich93}. 

For the sake of expediency, we prove a GAGA-type statement only for certain \'etale morphisms over a subfield of $\C$ equipped with the hybrid norm. 
To this end, we recall below the definition of the structure sheaf on the analytification of open subvarieties of affine space; it first appears in~\cite[\S 1.5]{berkovich}, and it is further developed in~\cite{poineau1,poineau2,lemanissier}.
}

\begin{definition}\label{structure sheaf}
Let $A$ be a geometric base ring, and let $U \subseteq \A^{n,\An}_A$ be an open subset. 
\begin{enumerate}[label=(\roman*)]
\item The ring $\mathcal{K}(U)$ of \emph{rational functions on $U$} is the localization
$$
\mathcal{K}(U) \coloneqq S_U^{-1} A[T_1,\ldots,T_n],
$$
where $S_U = \{ p \in A[T_1,\ldots,T_n] \colon p \not\in \ker(x) \textrm{ for all $x \in U$} \}$ is the multiplicative subset of polynomials that do not vanish on $U$.
\item The ring $\O(U)$ of \emph{analytic functions on $U$} are those functions
$$
f \colon U \to \bigsqcup_{x \in U} \H(x)
$$
such that for any $x \in U$, the following two conditions hold:
\begin{enumerate}
\item $f(x) \in \H(x)$; 
\item there is an open neighbourhood $V \subseteq U$ of $x$, and a sequence $(p_i)_{i \in \Z_{\geq 0}} \subseteq \mathcal{K}(V)$ such that $p_i \to f$ uniformly on $V$.
\end{enumerate}
\end{enumerate}
\end{definition}

\spa{The rings of~\cref{structure sheaf}(ii) give rise to a sheaf $\O_{\A^{n,\an}_A}$ of local rings on $\A^{n,\An}_A$ (in fact, one can make this definition for any Banach ring $A$, as explained in~\cite[\S 1.1.3]{poineau1}.
For a finite-type $A$-scheme $X$, the structure sheaf on the affine spaces can be used to construct the structure sheaf $\O_{X^{\An}}$ on $X^{\An}$ as follows: 
work affine-locally, so that $X$ admits an immersion into an affine space over $A$, which decomposes as
$$
X \stackrel{j}{\hookrightarrow} U \stackrel{i}{\hookrightarrow} \A^{n}_A,
$$
where $j$ is a closed immersion and $i$ is an open immersion. 
Then, $\O_{U^{\An}}$ is the restriction of $\O_{\A^{n,\An}_A}$ to $U^{\An}$, and $\O_{X^{\An}}$ is the quotient sheaf $\O_{U^{\An}}/ \mathcal{J}_X \cdot \O_{U^{\An}}$, where $\mathcal{J}_X \subseteq \O_U$ is the ideal sheaf of $X$ on $U$.

If $A = \C$ is equipped with the Archimedean norm, then this sheaf on $X^{\An} = X^h$ coincides with the usual sheaf of holomorphic functions on $X^h$. 
Further, if $A = \C$ is equipped with the hybrid norm, then $\O_{X^{\An}}$ can be related to $\O_{X^h}$ on the Archimedean fibres of $X^{\An}$, as is explained in the following comparison result.
%
%
}


\begin{lemma}\label{isomorphism local rings}
Let $k \subseteq \C$ be a subfield equipped with the hybrid norm,
and let $X$ be a finite-type $k$-scheme.
Suppose $k$ is not included in $\R$.
If $x \in \lambda^{-1}((0,1])$ corresponds to $(\rho,x') \in (0,1] \times X^h$, then there is a natural isomorphism 
\begin{equation}\label{nat isom local rings}
\O_{X^{\An},x} \stackrel{\simeq}{\longrightarrow} \O_{X^h,x'}
\end{equation}
of local rings. 
\end{lemma}

\begin{proof}
The question is local, so we may assume $X$ is affine.
In this case, $X$ admits an immersion into an affine space over $k$. 
Thus, it suffices to produce the isomorphism when $X$ admits a closed or open immersion into $\A^{n}_k$, with the latter case being trivial.
If $X$ is closed in $\A^{n}_k$, then both rings in~\ref{nat isom local rings} are quotients of the corresponding local rings of $\A^{n,\An}_k$ and $\A^{n,h}_k$ by the extension of the ideal defining $X$, so it suffices to exhibit the isomorphism~\ref{nat isom local rings} when $X = \A^{n}_k$.

Now, assume $X = \A^n_k$. 
For an open neighbourhood $V \subseteq X^{\An}$ of $x$, let $V' \subseteq X^h$ be the neighbourhood of $x'$ corresponding to $V \cap \lambda^{-1}(\rho)$ under the usual homeomorphism $\lambda^{-1}(\rho) \simeq X^h$.
For $f \in \O_{X^{\An}}$, write $f'$ for the map 
$$
V' \to \bigsqcup_{w \in V'} \H_{V'}(w)
$$
given by $w \mapsto f(w^{\rho})$, where we identify $\H_{V'}(w)$ and $\H_{V}(w^{\rho})$ as explained in~\cref{hybrid analytification}. 
This rule $f'$ defines an analytic function on $V'$, and the assignment $f \mapsto f'$ gives a ring homomorphism $\O_{X^{\An}}(V) \to \O_{X^h}(V')$.
As every sufficiently-small open neighbourhood of $x'$ arises as $V'$ for some open neighbourhood $V$ of $x$ in $X^{\An}$, there is an induced map $\O_{X^{\An},x} \to \O_{X^h,x'}$ on stalks.
The inverse map is constructed as follows: if $f'$ is an analytic function defined on an open neighbourhood $V'$ of $x'$, then write $V$ for open neighbourhood of $x$ corresponding to $(0,1] \times V'$ under the isomorphism $\lambda^{-1}((0,1]) \simeq (0,1] \times X^h$, and let $f$ be the pullback of $f'$ to $V$. 
Thus, we have an isomorphism as in \ref{nat isom local rings}.
\end{proof}

\begin{corollary}\label{local_inversion}
Let $k \subseteq \C$ be a subfield equipped with the hybrid norm, and suppose $k$ is not included in $\R$.
If $\varphi \colon Y \to X$ is an \'etale morphism between finite-type $k$-schemes, then $\varphi^{\An}$ is \'etale over the Archimedean fibres; in particular, $\varphi^{\An}$ is a local homeomorphism over the Archimedean fibres.
\end{corollary}


\begin{proof}
For any $y \in \lambda_Y^{-1}((0,1])$, \cref{isomorphism local rings} shows that $\O_{X^{\An},\varphi^{\An}(y)} \to \O_{Y^{\An},y}$ is \'etale if and only if the map $\O_{X^h,\varphi^{\An}(y)'} \to \O_{Y^h,y'}$ is \'etale, which holds by the classical GAGA theorem; see e.g. \cite[XII, 3.1]{sga1}.
By the inverse function theorem, $\varphi^h$ is a local isomorphism at $y'$; in particular, $\varphi^h$ is a local homeomorphism at $y'$.
This occurs if and only if $\varphi^{\An}$ is a local homeomorphism at $y$, as required.
\end{proof}

\spa{\label{counterexample 5}
If $k$ is a subfield of $\R$ equipped with the Archimedean norm $|\cdot |_{\infty}$, 
then the conclusion of~\cref{local_inversion} does not necessarily hold. For example, consider the morphism 
$\varphi \colon \Spec( \R[T^{\pm 1}]) \to \Spec(\R[T^{\pm 1}])$ given by $T \mapsto T^2$: $\varphi$ is \'etale but $\varphi^{\hyb}$ is not a local homeomorphism at the point $x = \ev(1,i)$ (where the $\ev$-notation is as in~\cref{hybrid affine line example}). 
Indeed, $\varphi^{\hyb}$ is not injective on any open neighbourhood of $x$ because $x_n = \ev(1,-1+i/n)$ has two distinct $\varphi^{\hyb}$-preimages for any $n \geq 1$, and $x_n \to \varphi^{\hyb}(x)$ as $n \to +\infty$.

For this reason, it will be useful in the sequel to pass from a field $k \subseteq \R$ to the extension field $k[i]$ of $k$, which is a subfield of $\C$ that is also equipped with the hybrid norm and it is no longer contained in $\R$. 
To this end, for a finite-type $k$-scheme $X$, write 
$$
\pi_{k[i]/k} \colon X^{\hyb}_{k[i]} \to X^{\hyb}
$$
for the ground field extension map; see~\cite[\S 4.3]{lemanissier} for a general discussion.
If $X = \Spec(B)$ is an affine $k$-scheme, then $\pi_{k[i]/k}$ is the continuous map given by restricting a seminorm on $B \otimes_k k[i]$ to $B$. 
In particular, $\pi_{k[i]/k}$ is $\lambda$-equivariant.
}

\spa{\label{complex conjugation 5}
In order to study the hybrid affine line $\A^{1,\hyb}_k$ by passing to the field extension $k[i]/k$, we can use action of $\mathrm{Gal}(k[i]/k)$ on $\A^{1,\hyb}_{k[i]}$ by complex conjugation.
More precisely, if $\iota \colon k[i][T] \to k[i][T]$ is the $k$-algebra morphism sending $i \mapsto -i$ and $T \mapsto T$, then the \emph{complex conjugation map} $I$ is the homeomorphism of $\A^{1,\hyb}_{k[i]}$ given by
$$
|f(I(x))| \coloneqq |(\iota (f))(x)|
$$
for $f \in k[i][T]$. It satisfies $\pi_{k[i]/k} \circ I = \pi_{k[i]/k}$, and $I \circ I$ is the identity map. Moreover, for any $p\in k[T]$, the morphism $\varphi_p \colon \A^{1,\hyb}_{k[i]} \to\ A^{1,\hyb}_{k[i]}$ given by $T \mapsto p$ satisfies $\varphi_p \circ I = I \circ \varphi_p$. 
In fact, $\A^{1,\hyb}_k$ is the (topological) quotient of $\A^{1,\hyb}_{k[i]}$ by the action of $I$, and $\pi_{k[i]/k}$ is the quotient map. 
In \S 3, this will be used as follows: 
given a homotopy
$$
H \colon [0,1]\times\A^{1,\hyb}_{k[i]} \to \A^{1,\hyb}_{k[i]}
$$
satisfying
$I(H(t,I(x)))=H(t,x)$ for $(t,x) \in [0,1]\times\A^{1,\hyb}_{k[i]}$, there exists a unique homotopy 
$$
H' \colon [0,1]\times\A^{1,\hyb}_k \to \A^{1,\hyb}_k
$$ 
such that $\pi_{k[i]/k} ( H ( t , x ) ) = H' ( t , \pi_{k[i]/k} ( x ) )$ for $(t,x) \in [0,1]\times\A^{1,\hyb}_{k[i]}$.
}

\subsection{Lifting Strong Deformation Retractions Along Finite Branched Covers}

\spa{In this section,
we prove a lifting theorem (\cref{lifting_homotopy}) for strong deformation retractions along a \emph{finite branched cover}, i.e.\ a finite open continuous map of topological spaces.
The principal example of interest is the hybrid analytification of a finite flat morphism between varieties, which is finite and flat by~\cref{finite flat gaga}, and hence open by~\cite[Corollaire 5.17]{lemanissier}.
This topological result, though elementary, is crucial to the proof of~\cref{theorem 12}.
}

\begin{lemma}\label{sepration_open}
Let $\varphi \colon Y \to X$ be a finite branched cover of locally compact, Hausdorff spaces. 
For $x\in X$, write $\varphi^{-1}(x)=\{y_1,\ldots,y_n\}$. 
There exists a basis $\mathcal U_x$ of open neighbourhoods of $x$ such that for any $U\in\mathcal{U}_x$, we have
$$
\varphi^{-1}(U) = \bigsqcup_{i=1}^n V_{i,U},
$$
where $V_{i,U}$ is an open neighbourhood of $y_i$ for each $i$.
Moreover, for each $i$, the set $\{ V_{i,U} \colon U \in \mathcal{U}_x \}$ is a basis of open neighbourhoods of $y_i$.
\end{lemma}

\begin{lemma}\label{lifting_path}
Let $\varphi \colon Y\to X$ be a finite branched cover of
locally compact, Hausdorff spaces.
Let $F$ be a closed subset of $X$, $U = X \backslash F$, and $\gamma \colon [0,1] \to X$ be a continuous map.
Suppose that
\begin{enumerate}[label=(\roman*)]
\item the restriction $\varphi\colon \varphi^{-1}(U)\to U$ is a local homeomorphism;
\item if $\gamma(t) \in F$ for some $t \in [0,1]$, then $\gamma(t)=\gamma(t')$ for all $t' \geq t$.
\end{enumerate}
For any $y\in\varphi^{-1}(\gamma(0))$, there is a unique continuous map $\tilde{\gamma} \colon [0,1] \to Y$ such that $\tilde{\gamma}(0)=y$ and $\tilde{\gamma}$ lifts $\gamma$ to $Y$,.
%
\end{lemma}
\begin{proof}
If $\gamma(t) \in U$ for all $t \in [0,1]$, then the result follows from the homotopy lifting property for covering maps. 
Otherwise, let $t_F \in [0,1]$ be the minimal $t \in [0,1]$ such that $\gamma(t) \in F$. 
If $t_F = 0$, then $\gamma$ is constant, and we can take $\tilde{\gamma}$ to be the constant map with value $y$.

Assume now that $t_F > 0$.
Write $\varphi^{-1}(\gamma(t_F)) = \{ y_1,\ldots,y_n \}$, and let $\mathcal{U}_{\gamma(t_F)}$ be the basis of open neighbourhoods of $\gamma(t_F)$ as in~\cref{sepration_open}.
Again using the homotopy lifting property, there is a unique lift $\tilde{\gamma}_1 \colon [0,t_F) \to Y$ of $\gamma |_{[0,t_F)}$ such that $\tilde{\gamma}_1(0) = y$.
For any $U \in \mathcal{U}_{\gamma(t_F)}$, there is $s_U \in [0,t_F)$ such that $\gamma( (s_U,t_F)) \subseteq U$; in particular, $\tilde{\gamma}_1((s_U,t_F)) \subseteq V_{i,U}$ for some $i$, which is independent of the choice of $U$. 
Define the map $\tilde{\gamma} \colon [0,1] \to Y$ by
$$
t \mapsto \begin{cases}
\tilde{\gamma}_1(t), & t \in [0,t_F)\\
y_i, & t \in [t_F,1]
\end{cases}
$$
The map $\tilde{\gamma}$ clearly lifts $\gamma$, and it remains to show that it is continuous at $t_F$.
For any open neighbourhood $V$ of $y_i$, there exists $U \in \mathcal{U}_{\gamma(t_F)}$ such that $V_{i,U} \subseteq V$, and so $\tilde{\gamma}^{-1}(V)$ contains the open neighbourhood $(s_U,1]$ of $t_F$, as required.
\end{proof}

\begin{proposition}\label{lifting_homotopy}
Let $\varphi\colon Y\to X$ be a finite branched cover of 
locally compact, Hausdorff spaces, $F$ be a closed subset of $X$, $U = X\backslash F$, and $H \colon [0,1] \times X \to X$ be a strong deformation retraction of $X$ onto $F$. 
Suppose that
\begin{enumerate}[label=(\roman*)]
\item the restriction $\varphi:\varphi^{-1}(U)\to U$ is a local homeomorphism ;
\item if $H(t,x) \in F$ for some $(t,x) \in [0,1] \times X$, then $H(t,x)=H(t',x)$ for all $t' \geq t$.
\end{enumerate}
There is a unique strong deformation retraction $\widetilde{H} \colon [0,1] \times Y \to Y$ of $Y$ onto $\varphi^{-1}(F)$ that lifts $H$ to $Y$.
\end{proposition}
\begin{proof}
By~\cref{lifting_path}, for any $y\in Y$, there exists a unique continuous map $\widetilde H(\cdot ,y)\colon[0,1]\to Y$ such that $\widetilde{H}(0,y)=y$ and that lifts $H(\cdot ,\varphi(y))\colon [0,1]\to X$. 
These can be combined to give a map $\widetilde{H} \colon [0,1] \times Y \to Y$, which clearly lifts $H$ to $Y$, i.e.\ $\varphi(\widetilde{H}(t,y)) = H(t,\varphi(y))$ for any $(t,y) \in [0,1] \times Y$.

It remains to show that $\widetilde{H}$ is continuous.
Fix $(t,y) \in [0,1] \times Y$. 
Write $\varphi^{-1}(H(t,\varphi(y))) = \{ z_1,\ldots,z_{\ell} \}$ and assume $z_1 = \widetilde{H}(t,y)$. 
Let $\mathcal{U}_{H(t,\varphi(y))}$ (resp. $\mathcal{U}_{\varphi(y)}$) be a basis of open neighbourhoods of $H(t,\varphi(y))$ (resp. $\varphi(y)$) as in~\cref{sepration_open}.
For $U \in \mathcal{U}_{H(t,\varphi(y))}$ (resp. $U'\in\mathcal{U}_{\varphi(y)}$), write $\varphi^{-1}(U) = \bigsqcup_{i=1}^{\ell} V_{i,U}$ (resp. $\varphi^{-1}(U') = \bigsqcup_{i=1}^{\ell} V_{i,U'}$). It suffices to show that for all $U \in \mathcal{U}_{H(t,\varphi(y))}$ sufficiently small $\widetilde{H}^{-1}(V_{1,U})$ contains a neighbourhood of $(t,y)$.

\begin{enumerate}
\item[Case 1.] Assume that $H(t,\varphi(y))\not\in F$. By the construction of $\widetilde{H}$, for any $U\in\mathcal{U}_{H(t,\varphi(y))}$ sufficiently small, there is a neighborhood $I \times V$ of $(t,y)$ and a continuous local section $\sigma \colon U \to V_{1,U}$ of $\varphi$ such that for any $(s,z) \in I \times V$, $\widetilde{H}(s,z) = \sigma(H(s,\varphi(z)))$. In particular, $\widetilde{H}^{-1}(V_{1,U})$ contains $I \times V$.
\item[Case 2.] Assume that $\varphi(y) \in F$, so
for any $s\in[0,1]$, $H(s,\varphi(y)) = \varphi(y)$. 
For any neighborhood $U$ of $\varphi(y)$, the continuity of $H$ guarantees that there exists a neighborhood $U'$ of $\varphi(y)$ such that $H([0,1]\times U')\subseteq U$. Write $V = V_{1,U'} \cap V_{1,U}$. 
We claim that $[0,1] \times V \subset \widetilde{H}^{-1}(V_{1,U})$. 
To this end, observe that for any $z \in V$, $\widetilde{H}(0,z) = z \in V_{1,U}$, the subset $\widetilde{H}([0,1],z)$ is connected subset of $\varphi^{-1}(U)$ by~\cref{lifting_path} and $\widetilde{H}([0,1],z)$ is contained in $\varphi^{-1}(U)=\bigsqcup_{i=1}^{\ell} V_{i,U}$. 
Thus, $\widetilde{H}([0,1],z)$ is a connected subset of $\bigsqcup_{i=1}^{\ell} V_{i,U}$ that has a non-empty intersection $V_{1,U}$, so $\widetilde{H}([0,1],z)$ must be included in $V_{1,U}$.

\item[Case 3.] Assume that $\varphi(y) \not\in F$ and $H(t,\varphi(y))\in F$. There is a smallest $t_1 \in [0,1]$ such that $\widetilde H(t_1,y)\in\varphi^{-1}(F)$. For any $s\geq t_1$, $\widetilde{H}(s,y)=\widetilde{H}(t_1,y)$; note that $t\geq t_1 > 0$ since $\varphi(y) \not\in F$ and $H(t,\varphi(y))\in F$.
For any $U\in\mathcal{U}_{H(t,\varphi(y))}$, there exists $U'\in\mathcal{U}_{\varphi(y)}$ and $t_2<t_1$ such that $H([t_2,1]\times U')\subseteq U$. 
After possibly shrinking $U'$ around $\varphi(y)$, we may assume that
$H(\{t_2\}\times U')\cap (X\backslash F)= \emptyset$. 
%
As in Case 1, the map $\widetilde{H}(t_2, \cdot )\colon V_{1,U'}\to \varphi^{-1}(U)$ is continuous, so there is a neighborhood $V \subseteq V_{1,U'}$ of $y$ such that $\widetilde{H}(\{t_2\}\times V)\subseteq V_{1,U}$. Now, as in Case 2, we can conclude that for any $z\in V$ $\widetilde{H}([t_2,1]\times\{z\})\subseteq V_{1,U}$. This implies that $\widetilde{H}^{-1}(V_{1,U})$ contains $[t_2,1]\times V$, as required.
\end{enumerate}
This completes the proof of the continuity of $\widetilde{H}$ at $(t,y)$.
\end{proof}

%


\section{The hybrid affine line over Archimedean fields}

%
\spa{
Let $k$ be a subfield of $\C$ equipped with the hybrid norm $\| \cdot \|_{\hyb} \coloneqq \max\{ |\cdot |_{\infty},| \cdot |_{0} \}$.
Let $\A^1_k = \Spec \left( k[T]\right)$ be the affine line over $k$, and let $\lambda \colon \A^{1,\hyb}_k \to [0,1]$ be the hybrid analytification as in~\cref{hybrid affine line example}.
Write $\A^{1,\triv}_k = \lambda^{-1}(0)$ for the non-Archimedean fibre, and $\A^{1,\arch}_k = \lambda^{-1}((0,1])$ for the Archimedean fibres.
For any subset $V \subseteq \A^{1,\hyb}_k$ and a positive constant $\delta > 0$, write $V(\delta)$ for the intersection of $V$ with the subset $\{ x \in \A^{1,\hyb}_k \colon \lambda(x) \leq \delta \}$; 
note that $V(\delta) = V$ for $\delta \geq 1$.
Similarly, write $V^{\triv} \coloneqq V(0) = V \cap \lambda^{-1}(0)$.
}

\subsection{The Local Structure of the Hybrid Affine Line}\label{local structure}

\spa{The goal of this section is to describe a cover of $\A^{1,\hyb}_k$ by compact subsets whose structure near
$\A^{1,\triv}_k$
is well understood.
The members of this cover will be \emph{hybrid closed discs} associated to a branch $[\eta_{p,0},\eta_{p,1})$ of $\A^{1,\hyb}_k$. The intersection of such a disc with $\A^{1,\triv}_k$ is $[\eta_{p,0},\eta_{p,1}]$, and the intersection with $\lambda^{-1}(1)$ is the union of the closed discs of radius $1$ centered at the roots of $p$. 

More generally, we will define a hybrid closed disc associated to a polynomial $p \in k[T]$.
%
}

\begin{definition}\label{hybrid closed disc}
Let $p \in k[T]$. 
The \emph{hybrid closed disc} $\D_k(p)$ around $p$ is the closure in $\A^{1,\hyb}_k$ of the subset
\begin{equation}\label{definition of disc}
\{ x\in\A^{1,\hyb}_k\colon |p(x)| < \chi(\lambda(x))\},
\end{equation}
where $\chi \colon \R_{\geq 0} \to \R_{\geq 0}$ denotes the continuous function
given by
$$
t \mapsto \chi(t) \coloneqq \begin{cases}
t^t, & t > 0\\
1, & t = 0.
\end{cases}
$$
For $z \in k$, write $\D_k(z) \coloneqq \D_k(T-z)$, and $\D_k \coloneqq \D_k(0)$.
For $\delta > 0$, write $\D_k(p,\delta ) \coloneqq \D_k(p)(\delta)$.
Note that if $t \in (0,1]$, then $\ev(w,t) \in \mathbf{D}_k$ if and only if $|w|_{\infty} \leq t$; in particular, the intersection of $\D_k$ with an Archimedean fibre is a closed disc in the usual sense, hence the name.

Further, define
\begin{equation}\label{definition of boundary of disc1}
\C_k(p) \coloneqq \D_k(p) \cap \{x\in\A^{1,\hyb}_k\colon |p(x)|=\chi(\lambda(x))\},
\end{equation}
and the subsets $\C_k(z)$, $\C_k$ and $\C_k(p,\delta)$ are defined analogously. 
Finally, for $\delta > 0$, set
\begin{equation}\label{definition of boundary of disc2}
\E_k(p,\delta) \coloneqq \C_k(p,\delta)\cup\{x\in\A^{1,\hyb}_k\colon \lambda(x)=\delta,\ |p(x)|\leq \chi(\lambda(x))\}.
\end{equation}
%
%
\end{definition}

\spa{\label{automorphism disc}
The hybrid closed disc $\D_k(p,\delta)$ can be related to $\D_k(\delta)$ as follows: if $\varphi_p \colon \A^1_k \to \A^1_k$ is the morphism given by $T \mapsto p(T)$, then the analytification $\varphi^{\hyb}_p$ sends $\D_k(p,\delta)$ onto $\D_k(\delta)$. 
In fact, if $p = T-z$ for some $z \in k$, then $\varphi^{\hyb}_p$ restricts to a homeomorphism from $\D_k(p,\delta)$ onto $\D_k(\delta)$.
}

\spa{\label{hybrid disc topology}
The space $\D_k(p)$ is compact, and the (topological) interior is given by
\begin{equation}\label{interior of disc}
\{x\in\A^{1,\hyb}_k\colon |p(x)|< \chi(\lambda(x))\}.
\end{equation}
It follows that if $p$ is a monic, irreducible polynomial and $x \in [\eta_{p,0},\eta_{p,1})$, then $\D_k(p,\delta)$ is a compact neighbourhood of $x$ in $\A^{1,\hyb}_k$ for any $\delta > 0$. 
}
\spa{
If $p$ is a monic, irreducible polynomial, then the intersection of $\D_k(p)$ with each fibre of $\lambda$ is a (union of) closed discs in the fibre. 
More precisely, 
\begin{enumerate}[label=(\alph*)]
    \item $\D_k(p) \cap \lambda^{-1}(0)$ is the interval $[\eta_{p,0},\eta_{p,1}]$;
    \item for $\rho > 0$, $\D_k(p) \cap \lambda^{-1}(\rho)$ is 
    the closed subset $\{ |p| \leq 1 \}$,
    where we have identified $\lambda^{-1}(\rho)$ with $\C$ or $\C/\lnot$ as in~\cref{hybrid affine line example}. In particular, $\D_k(p) \cap \lambda^{-1}(\rho)$ is a closed disc when $p$ is linear. 
\end{enumerate}
It follows that the intersections $\mathbf{E}_k(p) \cap \A^{1,\triv}_k$ and $\mathbf{C}_k(p) \cap \A^{1,\triv}_k$ consists only of $\{ \eta_{0,1} \}$.
}

%
%
%
%

\begin{lemma}\label{sec 3 nbhood basis}
Let $p \in k[T]$ be a monic, irreducible polynomial and $\delta > 0$. 
For $r \in [0,1]$, a basis of open neighborhoods of $\eta_{p,r}$ in $\D_k(p,\delta)$ is given by subsets of the form
$$
\D_k(p,\delta) \cap \{ x \in \A^{1,\hyb}_k \colon r-\epsilon < |p(x)| < r+\epsilon , \lambda(x) < \epsilon \}
$$
for $\epsilon \in (0,\min\{ r,\delta\})$.
%
\end{lemma}

\begin{proof}
It suffices to show that any open neighbourhood $U$ of $\eta_{p,r}$ in $\A^{1,\hyb}_k$ contains an open neighbourhood of the form $\{ r - \epsilon < |p| < r + \epsilon \}$ for some $\epsilon > 0$.
For any $\epsilon > 0$, we write
$$
V_\epsilon \coloneqq \D_k(p,\delta) \cap \{ x \in \A^{1,\hyb}_k \colon r-\epsilon \leq |p(x)| \leq r+\epsilon , \lambda(x) \leq \epsilon \},
$$
which is a compact neighbourhood of $\eta_{p,r}$ in $\D_k(p,\delta)$.
The intersection $\bigcap_{\epsilon > 0} V_\epsilon$ is equal to $\{\eta_{p,r}\}$, 
and hence it is included in $U$. 
We claim that there exists $\epsilon > 0$ such that $V_{\epsilon} \subseteq U$.
If not, then for any $\epsilon > 0$ there exists $x_{\epsilon} \in V_{\epsilon} \backslash U$. As $\D_k(p,\delta)$ is compact, the net $(x_{\epsilon})$ must have a convergent subnet with limit $x \in \D_k(p,\delta)$; as the $V_{\epsilon}$'s intersect in $\eta_{p,r}$, we must have $x = \eta_{p,r}$, a contradiction.
%
%
\end{proof}

%

\spa{
The homeomorphism-type of the hybrid closed discs can be understood in terms of the auxiliary construction below.
Consider the solid cylinder
$$
\DD \coloneqq \{ (z,t) \in \C \times [0,1] \colon |z|_{\infty} \leq 1 \},
$$
as well as the subset $\BB \coloneqq \{ (z,t) \in \DD \colon t = 0\}$ of $\DD$.
There is a continuous map $\nu \colon \BB \to [0,1]$ given by $(z,0) \mapsto |z|_{\infty}$, and 
let $\widetilde{\DD}$ be the quotient of $D \sqcup [0,1]$ by the equivalence relation $b \sim \nu(b)$ for $b \in B$; that is, 
$\widetilde{\DD}$ is the pushout of the inclusion $\BB \hookrightarrow \DD$ along $\nu$ 
in the category of topological spaces. 
Note that the space $\widetilde{\DD}$ is also compact, since it is a quotient of a compact space.

Furthermore, define a continuous map $\sigma \colon \DD \to\DD$ by $(z,t) \mapsto (\overline{z},t)$, where $\overline{z}$ denotes the complex conjugate of $z$. 
This descends to a continuous map $\widetilde{\sigma} \colon \widetilde{\DD} \to \widetilde{\DD}$, and let $\sim$ be the equivalence relation on $\widetilde{\DD}$ that it generates; that is, $x \sim y$ if and only if $\widetilde{\sigma}(x) = y$. 

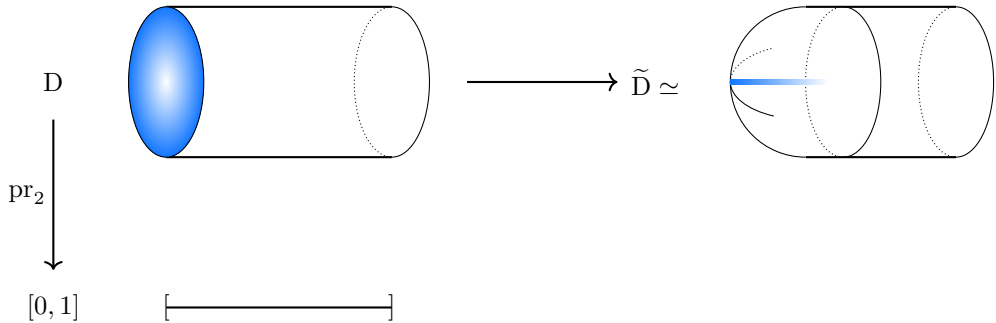
\begin{figure}[h]
  \centering
  \begin{tikzpicture}
    \draw (-1,-1) arc (-90:90:0.5cm and 1cm);
    \draw[densely dotted] (-1,1) arc (90:270:0.5cm and 1cm);
    \draw[thick] (-1,-1) -- (-4,-1);
    \draw[thick] (-1,1) -- (-4,1);
    \shadedraw[inner color=white,outer color=NavyBlue, draw=black] (-4,0) ellipse (0.5cm and 1cm);
    
    \draw[thick] (-4,-3) -- (-1,-3);
    \node at (-4,-3) {$[$};
    \node at (-1,-3) {$]$};
    \draw[thick,->] (-5.5,-0.5) -- (-5.5,-2.5);
    \node at (-5.5,0) {$\DD$};
    \node at (-5.5,-3) {$[0,1]$};
    \node at (-5.85,-1.5) {$\mathrm{pr}_2$};
    
    \draw[thick,->] (0,0) -- (2,0);
    \node at (2.5,0) {$\widetilde{\DD} \simeq$};

   \draw (4.5,0) circle (1cm);
    \draw (3.5,0) arc (180:245:1cm and 0.5cm);
    \draw[densely dotted] (3.5,0) arc (180:115:1cm and 0.5cm);
    \fill[white] (4.5,1.1) rectangle (6,-1.1);
    \draw (6.5,-1) arc (-90:90:0.5cm and 1cm);
    \draw[densely dotted] (6.5,1) arc (90:270:0.5cm and 1cm);
    \draw[thick] (4.5,1) -- (6.5,1);
    \draw[thick] (4.5,-1) -- (6.5,-1);
    \draw[densely dotted] (5,1) arc (90:270:0.5cm and 1cm);
    \shade[left color=NavyBlue, right color = white] (3.5,-0.03) -- (3.5,0.03) -- (4.8,0.03) -- (4.8,-0.03);
    \draw (5,-1) arc (-90:90:0.5cm and 1cm);
  \end{tikzpicture}
  \caption{The quotient map $\DD \twoheadrightarrow \widetilde{\DD}$ identifies the concentric circles in the fibre $\mathrm{B} = \mathrm{pr}_2^{-1}(0)$, giving rise to a topological space homeomorphic to a closed ball in $\R^3$. The subset $\mathrm{E} \subseteq \DD$, which consists of the cylindrical shell and the fibre $\pr_2^{-1}(1)$, is sent onto the boundary sphere of the closed ball under this homeomorphism.
  Note that the induced map $\widetilde{D} \to [0,1]$ is not the naive projection that one might envision from the figure.
  }
  \label{fig:hybrid_disc}
\end{figure}  

Finally, consider the subsets $\mathrm{C}\coloneqq\{ (z,t) \in \DD \colon |z|_{\infty}= 1 \}$ and $\mathrm{E} \coloneqq  \mathrm{C}\cup \{ (z,t) \in \DD \colon t=1 \}$. 
Write $\widetilde{\mathrm{C}}$ (resp.\ $\widetilde{\mathrm{E}}$) for the image of $\mathrm{C}$ (resp.\ of $\mathrm{E}$) under $\mathrm{D} \to \widetilde{\mathrm{D}}$.
The map $\sigma$ restricts to a homeomorphism of the subset $\mathrm{C}$ (resp.\ of $\mathrm{E}$), so we can form the quotient  $\widetilde{\mathrm{C}}/\lnot$ (resp.\  $\widetilde{\mathrm{E}}/\lnot$). 
This quotient coincides with the image of $\mathrm{C}$ (resp.\ $\mathrm{E}$) under the map $\widetilde{\DD} \to \widetilde{\DD}/\lnot$.
%
}



\begin{proposition}\label{sec3 prop 1}
Let $\delta>0$ and $z\in k$. 
\begin{enumerate}[label=(\roman*)]
\item If $k$ is not contained in $\mathbf{R}$, then there is a homeomorphism $f_{\delta} \colon \widetilde{\DD} \to \D_k(z,\delta)$ that restricts to a homeomorphism between $\widetilde{\mathrm{C}}$ and $\C_k(z,\delta)$, and between $\widetilde{\mathrm{E}}$ and $\E_k(z,\delta)$.
\item If $k$ is contained in $\mathbf{R}$, then there is a homeomorphism $f_{\delta} \colon\widetilde{\DD} / \lnot \to \D_k(z,\delta)$ that restricts to a homeomorphism between $\widetilde{\mathrm{C}}/\lnot$ and $\C_k(z,\delta)$, and between $\widetilde{\mathrm{E}}/\lnot$ and $\E_k(z,\delta)$.
\end{enumerate}
\end{proposition}

\begin{proof}
After applying the automorphism $T \mapsto T - z$ as in~\cref{automorphism disc}, we may assume that $z=0$.
Assume first that $k$ is not contained in $\R$. 
Define a map $g_{\delta} \colon \DD \to \D_k(\delta)$ by the formula
$$
(w,s) \mapsto \begin{cases}
\ev\left( w\delta s |w|_{\infty}^{\frac{1}{\delta s}-1},\delta s \right), & s\in (0,1],\\
\eta_{0,|w|_{\infty}}, & s = 0.
\end{cases}
$$
The restriction of $g_{\delta}$ to $\DD \backslash \BB$ is a homeomorphism
$$
\DD\backslash \BB \stackrel{g_{\delta}}{\longrightarrow} \{ \ev(w,s) \in \A^{1,\hyb}_k \colon w \in \C, s \in (0,1], |w|_{\infty} \leq s \leq \delta \}.
$$
Indeed, it is clearly continuous, and it has a continuous inverse given by
$$
\ev(w,s) \mapsto \left( w s^{-s} |w|_{\infty}^{s-1} , \delta^{-1} s \right).
$$
To show the continuity of $g_{\delta}$ at a point $(w,0) \in B$, assume $r \coloneqq |w|_{\infty} \in (0,1)$. 
By~\cref{sec 3 nbhood basis}, it suffices to show that the $g_{\delta}$-preimage of the open set 
\begin{align*}
V_{r,\epsilon} &= \{ x \in \A^{1,\hyb}_k \colon r-\epsilon < |T(x)| < r+\epsilon, \lambda(x) < \epsilon \} \cap \D_k(\delta) \\
&= (\eta_{0,r-\epsilon},\eta_{0,r+\epsilon}) \sqcup \{ \ev(u,s) \colon u \in \C, s \in (0,\epsilon), r-\epsilon < |u|_{\infty}^s < r+\epsilon \}
\end{align*}
is open, for $\epsilon \in (0, \min\{ r, 1-r, \delta \})$. 
Observe that $|T(g_{\delta}(u,s))| = \chi(\delta s) |u|_{\infty}$ for $(u,s) \in \DD$, and hence
$$
g_{\delta}^{-1}(V_{r,\epsilon}) = \{ (u,s) \in \DD \colon (r-\epsilon) \chi(\delta s)^{-1} < |u|_{\infty} < (r+\epsilon)\chi(\delta s)^{-1}, \delta^{-1} s < \epsilon \},
$$
which is open because $\chi$ is continuous and non-vanishing. 
The cases when $w =0$ or $|w|_{\infty} = 1$ are analogous.
Thus, $g_{\delta}$ is continuous everywhere.

Define a continuous map $h_{\delta} \colon [0,1] \to \D_k(\delta)$ by the formula $r \mapsto \eta_{0,r}$. 
By construction, $g_{\delta} |_B = h_{\delta} \circ \iota$, so there is a unique continuous map $f_{\delta} \colon \widetilde{\DD} \to \D_k(\delta)$ such that the diagram
\begin{center}
\begin{tikzcd}
B \arrow[hook]{r}{} \arrow[swap]{d}{{\nu }} & \DD \arrow{d}{} \arrow[bend left]{rdd}{g_{\delta}} \\
{[0,1]} \arrow{r}{} \arrow[bend right]{rrd}{h_{\delta}} & \widetilde{\DD} \arrow[dashed]{rd}{\exists ! f_{\delta}} & \\
 & & \D_k(\delta)
\end{tikzcd}
\end{center}
is commutative. 
As $\widetilde{\DD}$ is compact and $\D_k(\delta)$ is Hausdorff, it suffices to show that $f_{\delta}$ is bijective.
As a set, $\widetilde{\DD}$ is a disjoint union of $[0,1]$ and $\DD \backslash \mathrm{B}$, and the restriction of $f_{\delta}$ to each is $h_{\delta}$ and $g_{\delta}$, respectively. Both restrictions are bijections onto their images, and the disjoint union of their images is all of $\D_k(\delta)$. Thus, $f_{\delta}$ is bijective, as required. 

Furthermore, it is easy to check that $f_{\delta}$ restricts to a homeomorphism of $\widetilde{\mathrm{C}}$ onto $\mathbf{C}_k(\delta)$, and of $\widetilde{\mathrm{E}}$ onto $\mathbf{E}_k(\delta)$. 
Similarly, when $k$ is included in $\mathbf{R}$, the same construction holds after precomposing $g_{\delta}$ with the quotient $\mathrm{D} \to \mathrm{D}/\lnot$ of $\mathrm{D}$ under complex conjugation.
\end{proof}

\spa{\label{D tilde is a ball}
As illustrated in~\cref{fig:hybrid_disc}, the space $\DD$ is a closed solid cylinder, and $\mathrm{E}$ is the boundary of $\DD$ 
with the open disc $\{ (z,t) \in \DD \colon |z|_{\infty}< 1, t=0\}$ removed.
For each $r \in [0,1]$, the quotient map $\DD \to \widetilde{\DD}$ identifies
all the points on the circle $\{ (z,t) \in \DD \colon |z|_{\infty}= r, t=0\}$;
it follows that 
$\widetilde{\DD}$ is homotopic to a closed ball in $\R^3$, and the inclusion $\widetilde{\mathrm{E}}\hookrightarrow \widetilde{\DD}$ is homotopic to the inclusion of the boundary sphere of the closed ball. 
In particular, after removing the image of the point $(z,t) = (0,0) \in \DD$ from $\widetilde{\DD}$, then the inclusion of $\widetilde{\mathrm{E}}$ becomes a homotopy equivalence.
Since $\widetilde{\mathrm{E}}$ is not contractible, it follows that this inclusion cannot be homotopic to a constant map. 
Applying the homeomorphism $f_{\delta}$ to this observation yields the following corollary of~\cref{sec3 prop 1}.
}

\begin{corollary}\label{corollary homotopy disc}
If $k$ is not contained in $\mathbf{R}$ and $z \in k$, then the inclusion $\E_k(z) \hookrightarrow \D_k(z) \backslash \{ \eta_{z,0}\}$ is not homotopic to a constant map. 
\end{corollary}

We will later use~\cref{corollary homotopy disc} to understand the possible homotopies of the inclusion map $\D_k(z) \hookrightarrow \A^{1,\hyb}_k$ for $z \in k$; see~\cref{non contractibility corollary 2}.
This is one of the key ingredients to~\cref{theorem 1}(2).

\begin{remark}
By arguing as in~\cite{poineau2}, one can show that the hybrid affine line $\A^{1,\An}_k$ is Fr\'echet--Urysohn and angelic when $k$ is not contained in $\R$. 
This can be done by exploiting the homeomorphism between a hybrid closed disc $\D_k(z)$ and a closed ball in $\R^3$, which exists by~\cref{sec3 prop 1}(i) and the observation in~\cref{D tilde is a ball}.
\end{remark}

\spa{\label{def of boundary}
The subset $\BB \cup \mathrm{C}$ of $\DD$ is homeomorphic to a cylindrical shell that is capped at one end and open at the other, and so one can imagine a strong deformation retraction of $\DD$ onto $\BB \cup \mathrm{C}$ obtained by ``pushing'' the interior of the cylinder towards the walls and the cap. 
This informal picture can, using~\cref{sec3 prop 1}, be used to construct strong deformation retractions of a hybrid closed disc onto certain subsets, namely
$$
\NNN_k(p,\delta)\coloneqq \C_k(p,\delta)\cup\D_k(p)^{\mathrm{triv}},
$$
and
$$
\NNN^+_k(p,\delta)\coloneqq \C_k(p,\delta)\cup\D_k(p)^{\mathrm{triv}}\cup\{x\in\A^{1,\hyb}_k\colon p(x)=0,\ \lambda(x)\leq \delta\}.
$$
The strong deformation retraction is constructed explicitly when $p = T$ in~\cref{example retraction} below, and the general case is~\cref{corollary retraction disc}.
}

\begin{example}\label{example retraction}
For any $\delta > 0$, there is a strong deformation retraction of $\D_k(\delta)$ onto $\mathbf{N}_k^+(\delta)$. 
By~\cref{sec3 prop 1}, it suffices to construct a strong deformation retraction of $\mathrm{D}$ onto $\mathrm{B} \cup \mathrm{C} \cup \{ (w,s) \in \mathrm{D} \colon w = 0 \}$.
Consider the homotopy $H \colon [0,1] \times \mathrm{D} \to \mathrm{D}$ given by
$$
\begin{array}{cccc}
     ( t , (w , s)) & \mapsto & \left\{\begin{array}{ll}
          ( 0 , s ) &\text{if } w = 0\\
          \left(\min\left(|w|,\dfrac{1-t}{4}\right)\dfrac{w}{|w|},\min(s,1-t+s-4|w|)\right)& \text{if }0 < |w| \leq \dfrac{s}{4}\text{ and }s \leq 1 \\
          \left(\left(\left(\dfrac{1}{2}-\dfrac{1-2|w|}{2-s}\right)\left(1-\dfrac{\min(s,1-t)}{2}\right)+\dfrac{1}{2}\right)\dfrac{w}{|w|},\min(s,1-t)\right)& \begin{array}{l}
                 \text{if }0 \leq s \leq 1,\ w \neq 0 \text{ and }\\
                 \dfrac{s}{4} \leq |w| \leq 1 - \dfrac{s}{4}
          \end{array}
           \\
          \left(\max\left(|w|,\dfrac{t+3}{4}\right)\dfrac{w}{|w|},\max(s,t+3+s-4|w|)\right)& \text{if }1 - \dfrac{s}{4} \leq |w| \leq 1\text{ and }s \leq 1
     \end{array}\right.
\end{array}
$$
Observe 
that for any $( t , (w , s) ) \in [ 0 , 1 ] \times \mathrm{D}$, we have $\sigma(H(t,\sigma(w,s)))=H(t,(w,s))$, where recall that $\sigma \colon \mathrm{D} \to \mathrm{D}$ is the complex conjugation given by $(w,s) \mapsto (\overline{w},s)$. 

\begin{figure}[h]
  \centering
  \begin{tikzpicture}
    \draw[outer color=NavyBlue] (-1,-2) arc (-90:90:1cm and 2cm);
    \draw[densely dotted,outer color=NavyBlue] (-1,2) arc (90:270:1cm and 2cm);
    \draw[densely dotted,inner color=white] (-1,0) ellipse (.8cm and 1.6cm);
    \draw[densely dotted,inner color=NavyBlue, outer color=NavyBlue] (-1,0) ellipse (.2cm and .4cm);
    \draw[densely dotted,outer color=NavyBlue] (-3,0) ellipse (.66cm and 1.33cm);
    \draw[densely dotted,inner color=white] (-3,0) ellipse (.2cm and .66cm);
    
    \draw[densely dotted] (-3,1.33) -- (-1,1.6);
    \draw[densely dotted] (-3,-1.33) -- (-1,-1.6);
    \draw[densely dotted] (-2.4,.55) -- (-1,.4);
    \draw[densely dotted] (-2.4,-.55) -- (-1,-.4);
    
    \draw[thick] (-1,-2) -- (-6,-2);
    \draw[thick] (-1,2) -- (-6,2);
    \shadedraw[inner color=white,outer color=NavyBlue, draw=black] (-6,0) ellipse (1cm and 2cm);
  \end{tikzpicture}
  \caption{
  An illustration of the strong deformation retraction $H$ of the cylinder $\mathrm{D}$ onto the subset $\mathrm{B} \cup \mathrm{C} \cup \{ (w,s) \in \mathrm{D} \colon w = 0 \}$ presented in~\cref{example retraction}.
  }
  \label{fig:retraction_hybrid_disc}
\end{figure}
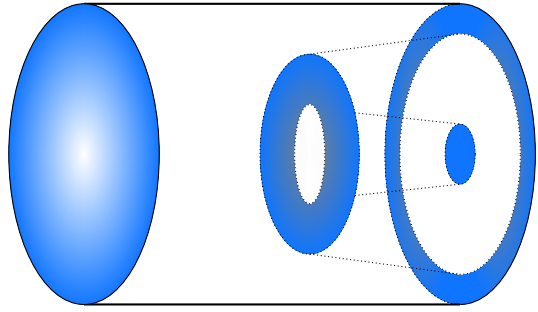  
\end{example}

\begin{proposition}\label{corollary retraction disc}
Let $p \in k[T]$ be a 
monic, separable
polynomial. 
There exists $\Delta>0$ such that for any $\delta\in(0,\Delta]$, there exists:
\begin{enumerate}[label=(\roman*)]
\item a strong deformation retraction of $\D_k(p,\delta)$ onto $\NNN_k(p,\delta)$;
\item a strong deformation retraction of $\D_k(p,\delta)$ onto $\NNN^+_k(p,\delta)$.
\end{enumerate}
Furthermore, if $p$ is linear, then we can take $\Delta = 1$.
\end{proposition}

\begin{proof}
First, remark that it suffices to show the retraction in (ii). 
Indeed, to construct the strong deformation retraction of (i), we may postcompose the homotopy in (ii) with a strong deformation retraction 
$$
G \colon [0,1] \times \mathbf{N}_k^+(p,\delta) \to \mathbf{N}_k^+(p,\delta)
$$
of $\mathbf{N}_k^+(p,\delta)$ onto $\mathbf{N}_k(p,\delta)$.
Such a map $G$ can be constructed as follows: if $\{ z_1,\ldots,z_{m} \} \subseteq \C$ are the roots of $p$, then one can take $G |_{[0,1] \times \mathbf{N}_k(p,\delta)}$ to be the identity and
$$
G(t,\ev(z_i,s)) = \begin{cases}
\ev(z_i, st) & t > 0,\\
\eta_{z_i,0} & t = 0,
\end{cases}
$$
for $t \in [0,1]$ and $s \in (0,\delta]$.
It is easy to verify that the map $G$ is continuous.

Thus, let us show (ii). 
If $\{ y_1,\ldots,y_{\ell}\} \subseteq \C$ are the roots of the derivative $p'$ in $\C$, then $p(y_i) \not=0$ for all $i$, since $p$ is separable.
In particular, there exists a threshold $\Delta > 0$ such that for any $\delta \in (0,\Delta]$, we have
$$
\D_k(p,\delta) \cap \{ x \in \A^{1,\hyb}_k \colon p'(x) = 0 \} = \emptyset.
$$
For example, we can take $\Delta$ less than the minimum (Archimedean) distance between a root of $p$ and of $p'$.

Assume first that $k$ is not included in $\mathbf{R}$. Fix $\delta \in (0,\Delta]$. 
Set $V \coloneqq \A^1_k \backslash \{ y_1,\ldots,y_{\ell} \}$, so that $\D_k(p,\delta) \subseteq V^{\hyb}$.
If $\varphi_p \colon \A^1_k \to \A^1_k$ is the morphism given by $T \mapsto p(T)$, then the restriction $\varphi_p |_V$ is \'etale. 
It follows from~\cref{local_inversion} that the analytification 
$$
\phi \coloneqq (\varphi_p |_V)^{\hyb} \colon V^{\hyb} \to \A^{1,\hyb}_k
$$
is a local homeomorphism over the Archimedean fibres of $\A^{1,\hyb}_k$, and hence over $U \coloneqq \D_k(\delta) \backslash \mathbf{N}^+_k(\delta)$. 
Let $F \coloneqq \D_k(\delta) \backslash U$ and let $H \colon [0,1] \times \D_k(\delta) \to \D_k(\delta)$ be a strong deformation retraction of $\D_k(\delta)$ onto $\mathbf{N}^+_k(\delta)$ as in~\cref{example retraction}. 
It is easy to verify that $\D_k(p,\delta) = \phi^{-1}(\D_k(\delta))$ and $\mathbf{N}_k^+(p,\delta) = \phi^{-1}(\mathbf{N}_k^+(\delta))$, and so~\cref{lifting_homotopy} shows that $H$ lifts to a strong deformation retraction $\widetilde{H}$ of $\D_k(p,\delta)$ onto $\mathbf{N}_k^+(p,\delta)$, which completes the proof of (ii).


Suppose now that $k$ is included in $\R$. 
Consider the ground field extension map $\pi_{k(i)/k} \colon \A^{1,\hyb}_{k[i]} \to \A^{1,\hyb}_k$ as in~\cref{counterexample 5}. 
The subset $\pi_{k[i]/k}^{-1}(\D_k(p,\delta))$ is equal to $\D_{k[i]}(p,\delta)$; note that $p$ may no longer be irreducible in $k[i][T]$.
As in the previous case, we may construct a strong retraction deformation $\widetilde{H}$ of $\D_{k[i]}(p,\delta)$ onto $\mathbf{N}_{k[i]}^+(p,\delta)$.
Now, to produce a strong retraction deformation of $\D_k(p,\delta)$ onto $\mathbf{N}_k^+(p,\delta)$ following~\cref{complex conjugation 5}, it suffices to demonstrate that  
$\widetilde{H}(t,x)=I\left(\widetilde{H}(t,I(x))\right)$ for $(t,x)\in[0,1]\times \D_{k[i]}(p,\delta)$.

By construction, the strong retraction deformation $H$ of~\cref{example retraction} satisfies this property. 
Observe that
for $(t,x)\in [0,1]\times \D_{k[i]}(p,\delta)$ we have
$$
\phi(I(\widetilde{H}(t,I(x))))=I(\phi(\widetilde{H}(t,I(x))))=I(H(t,\phi(I(x)))=H(t,x),
$$
where we have repeatedly used that $\phi \circ I = I \circ \phi$ (this holds since $p$ has coefficients in $k$).
Thus, 
the homotopy $(t,x)\mapsto I(\widetilde{H}(t,I(x)))$ is a lift of the  homotopy $H$ along $\phi$. 
The uniqueness of the lift in~\cref{lifting_homotopy} guarantees that $I(\widetilde{H}(t,I(x)))=\widetilde{H}(t,x)$ for all $(t,x)$, as required.
\end{proof}

\begin{corollary}\label{sec 3 prop 2 v2}
Let $p \in k[T]$ be a monic, irreducible polynomial.
There exists $\Delta > 0$ such that for any $\delta \in (0,\Delta]$, there exists:
\begin{enumerate}[label=(\roman*)]
\item a strong deformation retraction of $\A^{1,\hyb}_k(\delta)$ onto $\A^{1,\hyb}_k(\delta) \backslash \left( \D_k(p,\delta) \backslash  \NNN_k(p,\delta) \right)$;
\item a strong deformation retraction of $\A^{1,\hyb}_k(\delta)$ onto $\A^{1,\hyb}_k(\delta) \backslash \left( \D_k(p,\delta) \backslash  \NNN^+_k(p,\delta) \right)$;
\end{enumerate}
\end{corollary}

\begin{proof}
Let $\Delta > 0$ be as in~\cref{corollary retraction disc}, $\delta \in (0,\Delta]$, and let $H \colon [0,1] \times \D_k(p,\delta) \to \D_k(p,\delta)$ be a strong deformation retraction of $\D_k(p,\delta)$ onto $\NNN_k(p,\delta)$.
Consider the map $\widetilde{H} \colon [0,1] \times \A^{1,\hyb}_k(\delta) \to \A^{1,\hyb}_k(\delta)$ given by
$$
(t,x) \mapsto \begin{cases}
x & x \not\in \D_k(p,\delta),\\
H(t,x) & x \in \D_k(p,\delta).
\end{cases}
$$
The map $\widetilde{H}$ is a strong deformation retraction of the required form, provided it is continuous.
Observe that $(\A^{1,\hyb}_k(\delta) \backslash \D_k(p,\delta)) \cup \mathbf{C}_k(p,\delta)$ and $\mathbf{D}_k(p,\delta)$ form a closed cover of $\A^{1,\hyb}_k(\delta)$ with intersection equal to $\mathbf{C}_k(p,\delta)$, and $\widetilde{H}$ is continuous on each piece and the identity map on the intersection.
Thus, $\widetilde{H}$ is continuous.
\end{proof}

\spa{\label{variation on corollary}
If $p \in k[T]$ is reducible, then a variation on~\cref{sec 3 prop 2 v2} still holds: if $p = p_1\cdots p_N$ is a decomposition of $p$ into a product of monic irreducibles, then there is a threshhold $\Delta > 0$ such that for any $\delta \in (0,1]$, there is a strong deformation of $\A^{1,\hyb}_k(\delta)$ onto 
$$
\A^{1,\hyb}_k(\delta) \backslash \bigcup_{i=1}^N \left( \D_k(p_i,\delta) \backslash \mathbf{N}^+_k(p_i,\delta) \right).
$$
Indeed, pick $\Delta > 0$ small enough so that the interiors $\mathrm{int}\left( \D_k(p_i,\Delta) \right)$ are disjoint, and (after possibly shrinking $\Delta$) use the strong deformation retraction of~\cref{corollary retraction disc}(ii) for each disc $\D_k(p_i,\delta)$. 
}



\spa{\label{test label}
The \emph{hybrid closed disc at infinity}, denoted $\mathbf{D}_k(\infty)$, is the closure in $\A^{1,\hyb}_k$ of the subset  
$$
\{ x \in \A^{1,\hyb}_k \colon |T(x)| > \chi(\lambda(x))^{-1} \},
$$
where $\chi$ is as in~\cref{hybrid closed disc}. Observe that $\mathbf{D}_k(\infty)^{\mathrm{triv}} = [\eta_{0,1}, \eta_{0,\infty})$, and the hybrid closed discs $\D_k(p)$ (as $p$ ranges over all monic, irreducible polynomials in $k[T]$) and $\D_k(\infty)$ together form a closed cover of $\A^{1,\hyb}_k$.
By analogy with~\cref{definition of boundary of disc1} and~\cref{def of boundary}, for $\delta > 0$, set
$$
\mathbf{C}_k(\infty,\delta) \coloneqq \D_k(\infty,\delta) \cap \{ x \in \A^{1,\hyb}_k \colon |T(x)| = \chi(\lambda(x))^{-1} \},
$$
$$
\mathbf{N}_k(\infty,\delta) \coloneqq \C_k(\infty,\delta) \cup \D_k(\infty)^{\mathrm{triv}}.
$$
The following proposition is the analogue of~\cref{corollary retraction disc} and~\cref{sec 3 prop 2 v2} for the hybrid closed disc at infinity.
}

\begin{proposition}\label{def retraction at infinity}
For any $\delta \in (0,1]$, there exists:
\begin{enumerate}[label=(\roman*)]
    \item a strong deformation retraction of $\D_k(\infty,\delta)$ onto $\mathbf{N}_k(\infty,\delta)$;
    \item a strong deformation retraction of $\A^{1,\hyb}_k(\delta)$ onto $\A^{1,\hyb}_k(\delta) \backslash \left( \D_k(\infty,\delta) \backslash \mathbf{N}_k(\infty,\delta) \right)$.
\end{enumerate}
\end{proposition}

\begin{proof}
By arguing as in~\cref{sec 3 prop 2 v2}, it suffices to construct the retraction in (i). 
Now, fix $\delta \in (0,1]$ and consider the automorphism $\psi$ of the torus $U = \{ T \not= 0 \} \subseteq \A^{1}_k$ given by $T \mapsto T^{-1}$.
It is easy to see that $\psi^{\hyb}$ sends $\D_k(0,\delta) \cap U^{\hyb}$ and $\mathbf{N}_k(0,\delta) \cap U^{\hyb}$ homeomorphically onto $\D_k(\infty,\delta)$, and $\mathbf{N}_k(\infty,\delta)$, respectively. 
Thus, it suffices to construct a strong deformation retraction of $\D_k(0,\delta) \cap U^{\hyb}$ onto $\mathbf{N}_k(0,\delta) \cap U^{\hyb}$. 
Moreover, by~\cref{sec3 prop 1}, $\D_k(0,\delta) \cap U^{\hyb}$ is homeomorphic to the image of $\{ (z,t) \in \DD \colon z \not= 0 \}$ under the quotient map $\DD \to \widetilde{\DD}$, and $\mathbf{N}_k(0,\delta) \cap U^{\hyb}$ is sent to the image of 
\begin{equation}\label{subset of script D}
\{ (z,t) \in \DD \colon t = 0 \} \cup \{ (z,t) \in \DD \colon |z|_{\infty} = 1 \}
\end{equation}
under $\DD \to \widetilde{\DD}$. 
In particular, it is enough to find a strong deformation retraction of $\{ (z,t) \in \DD \colon z \not= 0 \}$ onto the subspace in~\cref{subset of script D}, and such a map clearly exists.
\end{proof}


\subsection{Contractibility of the Hybrid Affine Line over a Countable Archimedean Field}

\spa{
When $k$ is a countable subfield of the complex numbers equipped with the hybrid norm, we show below that the hybrid affine line $\A^{1,\hyb}_k$ admits a strong deformation retraction onto the non-Archimedean fibre; this is stated as~\cref{theorem 1}(1) in the introduction. 
The construction proceeds by patching together the strong deformation retractions of~\cref{sec 3 prop 2 v2} and~\cref{def retraction at infinity}.
}

\spa{
Let $k \subseteq \C$ be any subfield and set $X = \A^{1,\hyb}_k$.
For any $0 < \delta' \leq \delta \leq 1$, there is a strong deformation retraction $R \colon [0,1] \times X(\delta) \to X(\delta)$ of $X(\delta)$ onto $X(\delta')$.
%
For example, one can define a map $R$ at a point $(s,x) \in [0,1] \times X(\delta)$ by the formula
\begin{equation}\label{def retr onto smaller piece}
R(s,x) \coloneqq \begin{cases}
\ev(z, \min\{ t, \delta (1-s)+\delta' s \}), & x = \ev(z,t) \in \lambda^{-1}((0,\delta]),\\
x, & x \in \lambda^{-1}(0).
\end{cases}
\end{equation}
It is easy to verify that $R$ is continuous and that it is indeed a strong deformation retraction onto $X(\delta')$. 
}

\spa{\label{homotopy notation}
Let $f, f' \colon Y \to Z$ and $g, g' \colon X \to Y$ be continuous maps between topological spaces.
If $F$ is a homotopy from $f$ to $f'$ and $G$ is a homotopy from $g$ to $g'$, write $F \lhd G$ for the homotopy $[0,1] \times X \to Z$ between $f \circ g$ and $f' \circ g'$ given by the formula
$$
(F \lhd G)(t,x) \coloneqq \begin{cases}
F(0,G(2t,x)), & (t,x) \in [0,\frac{1}{2}] \times X,\\
F(2t-1,G(1,x)), & (t,x) \in [\frac{1}{2},1] \times X.
\end{cases}
$$
Beware that the above operation is not associative; for this reason, we adopt the notation
$$
F_n \lhd F_{n-1} \lhd \ldots \lhd F_2 \lhd F_1 \coloneqq  \left(\ldots \left( \left( F_n \lhd F_{n-1} \right) \lhd F_{n-2} \right) \ldots \lhd F_2 \right) \lhd F_1
$$
where $F_1,\ldots,F_n$ is a compatible collection of homotopies as above. 
This convention is chosen so that
\begin{equation}\label{iteration}
(F_n \lhd \ldots \lhd F_1)(1,\cdot) = (F_{n+1} \lhd F_n \lhd \ldots \lhd F_1)(1-2^{-n},\cdot)
\end{equation}
for any $n \geq 1$.
Moreover, it is clear that $F_n \lhd \ldots \lhd F_1$ is a strong deformation retraction if all $F_i$'s are so. 
}

\begin{theorem}\label{sec3 contractibility 2}
If $k$ is a countable subfield of $\C$, then there is a strong deformation retraction of $\A^{1,\hyb}_k$ onto the non-Archimedean fibre $\A^{1,\triv}_k$; in particular, $\A^{1,\hyb}_k$ is contractible.
\end{theorem}

The strong deformation retraction of~\cref{sec3 contractibility 2} is constructed by gluing many of the strong deformation retractions constructed in~\cref{local structure}. 
This gluing procedure is summarized in the lemma below.

\begin{lemma}\label{gluing sdrs}
Let $X$ be a locally compact, Hausdorff space.
For all $n \in \Z_{\geq 0}$, let $H_n\colon[0,1]\times X\to X$ be a strong retraction deformation of $X$ onto its image. 
Let $Y\coloneqq\bigcap_{n\in\Z_{\geq 0}}H_n(1,X)$ and $y_0\in Y$.
 Let $\{ D_n \}_{n \in \Z_{\geq 0}}$ be an open cover of $X\backslash\{y_0\}$ such that $D_n\cap D_m\cap Y=\emptyset$ whenever $n \not= m$, and 
write $V_n\coloneqq H_n(1,\cdot)^{-1}(D_n)$. 
Suppose that the following conditions hold: for $n \in \Z_{\geq 0}$ and $(t,x) \in [0,1] \times X$,
\begin{enumerate}[label=(\roman*)]
    \item 
    if $t \geq 1-2^{-n}$ and $m > n$, then $H_{m}(t,x)=H_n(t,x)$;
    \item 
    $H_{n+1}(1-2^{-(n+1)},x)=H_n(1,x)$;
    \item 
    if $H_n(t,x)\in Y$, then 
    $H_n(t',x)=H_n(t,x)$ for any $t' \geq t$;
    \item 
    $H_n(1,X)\cap D_n\subseteq Y$;
    \item 
    if $m < n$ and 
    $t\in [1-2^{-(m+1)},1]$, then $H_n(t,\cdot)^{-1}(D_m)=H_m(1,\cdot)^{-1}(D_m)$;
    \item 
    if $t' \in [0,t]$, then $H_n(t,X)\subseteq H_n(t',X)$;
\end{enumerate}
Let $\mathcal{U}$ be a basis of open neighbourhoods of $y_0$ in $X$ with the property that for any open set $W\in\mathcal{U}$, there exists $n\in\Z_{\geq 0}$ such that $H_n(1,X)\backslash W\subseteq Y\cap(\bigcup_{m=0}^n D_m)$.
Finally, set
$$
\varphi(x) \coloneqq \begin{cases}
H_n(1,x), & x \in V_n,\\
y_0, & x \in X \backslash \bigcup_{n\geq 0} V_n.
\end{cases}
$$
and
$$
H(s,x) \coloneqq \begin{cases}
H_n(s,x), & (s,x) \in [0,1-2^{-n}] \times X,\\
\varphi(x), & s=1.
\end{cases}
$$
Then, the functions $\varphi$ and $H$ are well-defined and continuous, and $H$ is a 
strong retraction deformation of $X$ onto $Y$.
\end{lemma}

\begin{proof}
In order to see that $\varphi$ is well-defined, it suffices to show that $V_n \cap V_m = \emptyset$ for $n \not= m$. 
Assume $m < n$.
By the assumption (iv), we have $H_n(1,\cdot)^{-1}(D_n) = H_n(1,\cdot)^{-1}(D_n \cap Y)$.
Using (v), we get that
$$
V_n \cap V_m = H_n(1,\cdot)^{-1}(D_n \cap Y) \cap H_n(1,\cdot)^{-1}(D_m) = H_n(1,\cdot)^{-1}(D_n \cap D_m \cap Y) = \emptyset,
$$
where the final equality holds by the assumption on the $D_n$'s.
Thus, $\varphi$ is well-defined. 
This implies, along with the assumption (i), that $H$ is also well-defined.

Now, we will demonstrate that $H$ is continuous. The proof is divided into three cases.
\begin{itemize}
\item[Case 1.] If $(t,x) \in [0,1) \times X$, then pick $n \in \Z_{\geq 0}$ such that $1-2^{-n} > t$. The restriction of $H$ to the open neighbourhood $[0,1-2^{-n}) \times X$ of $(t,x)$ coincides with $H_n$, and hence $H$ is continuous at $(t,x)$.

\item[Case 2.] If $x \in V_n$ for some $n$, 
then $H_n(1,x) \in Y$ by the assumption (iv).
By (ii) and (iii), for any $m > n$, we have
$$
H_m(1-2^{-(n+1)},x) = H_{n+1}(1-2^{-(n+1)},x) = H_n(1,x).
$$
The assumption (iii) implies that $H_m(s,x) = H_m(1-2^{-(n+1)},x) = H_n(1,x)$ for any $s \in [1-2^{-(n+1)},1]$. 
In particular, $H(s,x) = H_n(1,x)$ for $(s,x) \in [1-2^{-(n+1)},1] \times V_n$.
Thus, the restriction of $H$ to the open neighbourhood $(1-2^{-(n+1)},1] \times V_n$ of $(1,x)$ is equal to $H_n(1,\cdot)$, hence continuous.

\item[Case 3.] Suppose $x \not\in \bigcup_{n \geq 0} V_n$, 
so that $H(1,x) = y_0$.
Pick an open neighbourhood $W \in \mathcal{U}$ of $y_0$.
By assumption, there exists $n \in \Z_{\geq 0}$ such that $H_n(1,X) \backslash W \subseteq Y \cap \left( \cup_{m=0}^n D_m \right)$.

For any $z \in X$, (ii) implies that $H_{n+1}(1-2^{-(n+1)},z) = H_n(1,z)$; in addition, (i) asserts that $H_k(1-2^{-(n+1)},z) = H_{n+1}(1-2^{-(n+1)},z)$ for $k > n$.
It follows that $H_k(1-2^{-(n+1)},X) \subseteq H_n(1,X)$ for $k > n$, and hence (vi) implies that 
\begin{equation*}\label{temp inclusion 2}
H_k(t,X) \subseteq H_n(1,X)
\end{equation*}
for $k > n$ and $t \in [1-2^{-(n+1)},1]$.
In particular, $H(t,X) \subseteq H_n(1,X)$ for such $t$.
It follows that
$$
H(t,X) \backslash W \subseteq Y \cap \left( \bigcup_{m=0}^n D_m \right)
$$
for all $t \in [1-2^{-(n+1)},1]$, and hence
\begin{equation}\label{temp inclusion 3}
H(t,\cdot)^{-1}(X \backslash W) \subseteq \bigcup_{m=0}^n H(t,\cdot)^{-1}(D_m).
\end{equation}
For each fixed $t \in [1-2^{-(n+1)},1]$, there is a $k = k(t) > n$ such that 
$$
H(t,\cdot)^{-1}(D_m) = H_k(t,\cdot)^{-1}(D_m) = H_m(1,\cdot)^{-1}(D_m) = V_m,
$$
where the second equality follows from (v). Combining this with~\cref{temp inclusion 3}, we see that
\begin{equation}\label{temp inclusion 4}
H(t,\cdot)^{-1}(X\backslash W) \subseteq \bigcup_{m=0}^n V_m
\end{equation}
for any $t \in [1-2^{-(n+1)},1]$.

Now, set $C \coloneqq H_{n+1}(1-2^{-(n+1)},\cdot)^{-1}(X \backslash W)$; by the definition of $H$, $C = H(1-2^{-(n+1)},\cdot)^{-1}(X\backslash W)$.
As $H_{n+1}(1-2^{-(n+1)},\cdot)$ is continuous, $C$ is closed.
We claim that for any $(t,z) \in [1-2^{-(n+1)},1] \times (X\backslash C)$, $H(t,z) \in W$. 
Suppose that there is a pair $(t,z)$ with $H(t,z) \not\in W$, then
$$
H(t,z) \in H(t,X) \backslash W \subseteq Y \cap \left( \bigcup_{m=0}^n D_m \right).
$$
Thus, the inclusion~\cref{temp inclusion 4} implies that $z \in \bigcup_{m=0}^n V_m$. In this case, $H_n(1,z) \in Y$ by (iv), and so (iii) implies that
$$
H(t,z) = H_{n+1}(1-2^{-(n+1)},z) = H_n(1,z).
$$
That is, $z \in C$, a contradiction.

Therefore, we conclude that for any $(t,z) \in [1-2^{-(n+1)},1] \times (X \backslash C)$, $H(t,z)$ lies in $W$. 
As $x \in X \backslash C$ by construction, this demonstrates the continuity of $H$ at $(1,x)$.
\end{itemize}

It remains to show that $H$ is a strong deformation retraction onto $Y$. 
For any $x \in X$, if $x \in V_n$ for some $n$, then $H(1,x) = H_n(1,x) \in D_n$, and hence $H(1,x)$ lies in $Y$ by the assumption (iv); if $x$ does not lie in any $V_n$, then $H(1,x) = y_0 \in Y$. 
In particular, $H(1,X) \subseteq Y$. 

For $x \in Y$ and $t \in [0,1]$, we must show that $H(t,x) = x$. If $t \in [0,1)$, then there exists $n \in \Z_{\geq 0}$ such that $H(t,x) = H_n(t,x)$, and $H_n(t,x) = x$ by the assumption (iii). 
Assume now that $t = 1$.
If $x \in V_n$ for some $n$, 
then $H(1,x) = H_n(1,x) = x$, again by (iii). 
On the other hand, if $x \not\in \bigcup_{n \geq 0} V_n$, then we claim that $x = y_0$, and hence is fixed by $H(1,\cdot)$.
Indeed, (iii) guarantees that $V_n \cap Y = D_n \cap Y$ for any $n$, and hence $\{ V_n \cap Y \}_{n \in \Z_{\geq 0}}$ is an open cover of $Y \backslash \{ y_0 \}$. As $x$ does not lie in any member of this cover, we must have $x = y_0$.
This completes the proof.
\end{proof}

\begin{proof}[Proof of~\cref{sec3 contractibility 2}]
Throughout the proof, write $X = \A^{1,\hyb}$ and $X^{\triv} = \A^{1,\triv}$.
Let $(p_n)_{n \geq 1}$ be an enumeration of the monic irreducible polynomials of $k[T]$. 
For each $n \geq 1$, let $\Delta_n > 0$ be the constant produced in~\cref{sec 3 prop 2 v2} for the polynomial $p_n$,
and set $\Delta_0 = 1$.
For $n \geq -1$, set
$$
e_{n} = \begin{cases}
\displaystyle \min_{-1 \leq i \leq n} \{ 2^{-n-2}, \Delta_i \}, & n \geq 0,\\
1, & n=-1.
\end{cases}
$$
Observe that $e_n \leq e_{n-1}$ for all $n \geq 0$; in particular, $X(e_{n}) \subseteq X(e_{n-1})$ and $X(e_{-1}) = X$. 
Consider the following sequences of strong deformation retractions:
\begin{itemize}
\item for $n \geq 0$, write $R_n \colon [0,1] \times X(e_{n-1}) \to X(e_{n-1})$ for the strong deformation retraction of $X(e_{n-1})$ onto the subspace $X(e_n)$, constructed in~\cref{def retr onto smaller piece};  
\item for $n \geq 1$, write $G_n \colon [0,1] \times X(e_n) \to X(e_n)$ for a strong deformation retraction of $X(e_n)$ onto the subspace $X(e_n)\backslash \left(\D_k(p_n,e_n) \backslash \mathbf{N}_k(p_n,e_n)\right)$, as in~\cref{sec 3 prop 2 v2}(i); 
\item write $G_{0} \colon [0,1] \times X(e_{0}) \to X(e_0)$ for a strong deformation retraction of $X(e_0)$ onto the subspace $X(e_0) \backslash \left( \D_k(\infty,e_0) \backslash \mathbf{N}_k(\infty,e_0) \right)$, as in~\cref{def retraction at infinity}.
\end{itemize}
For $n \geq 0$, set $K_n \coloneqq G_n \lhd R_n$ and $H_n \coloneqq K_n \lhd \ldots \lhd K_0$, where we follow the notation of~\cref{homotopy notation}. Note that $K_n$ and $H_n$ are strong deformations of $X(e_{n-1})$ and of $X$, respectively, onto the same subspace 
$$
X(e_n) \backslash \left( \left( \D_k(\infty,e_0) \backslash \mathbf{N}_k(\infty,e_0) \right) \cup \left( \bigcup_{i = 1}^{n} \D_k(p_i,e_i) \backslash \mathbf{N}_k(p_i,e_i) \right) \right)
$$
of $X(e_{n})$. 
In order to produce the desired strong deformation retraction, it suffices to show that the $H_n$'s satisfy the conditions of~\cref{gluing sdrs}, where $Y = X^{\triv}$, $y_0 = \eta_{0,1}$, and 
$$
D_n = \mathbf{D}_k(p_n,e_n) \backslash \mathbf{N}_k(p_n,e_n).
$$
The conditions (i), (iii), (iv), and (vi) are obvious from the construction of $H_n$, while (ii) and (v) follow immediately from~\cref{iteration}.
Finally, take $\mathcal{U}$ to be any basis of open neighbourhoods of $\eta_{0,1}$ in $X$; the intersection of any member of $\mathcal{U}$ with $X^{\triv}$ avoids at most finitely-many branches of $X^{\triv}$, as required.
\end{proof}

\begin{remark}\label{contractibility for zariski opens}
The proof of~\cref{sec3 contractibility 2} can be modified to show the following: if $Z \subseteq \P^1_k$ is a non-empty, finite subset of closed points, then there is a strong deformation retraction of $\P^{1,\An}_k$ onto $\P^{1,\triv}_k \cup Z^{\An}$. 
Note that $\P^{1,\triv}_k \cup Z^{\An}$ is contractible, since $Z^{\An}$ is a disjoint union of intervals each intersecting $\P^{1,\triv}_k$ in one point, and so $\P^{1,\triv}_k \cup Z^{\An}$ admits a strong deformation retraction onto the contractible space $\P^{1,\triv}_k$.
%

To see this, it suffices to show the assertion with $\P^1_k$ replaced with $\A^1_k$, so we can
write $Z = \{ p = 0 \}$ for some monic, irreducible polynomial $p \in k[T]$.
Pick an enumeration $(p_n)_{n \geq 1}$ of the monic, irreducible polynomials in $k[T]$ such that $p_1 = p$.
Now, proceed as in the proof of~\cref{sec3 contractibility 2} with the strong deformation retraction $G_1$ as in~\cref{sec 3 prop 2 v2}(ii), and all other $G_n$'s as before.
%
\end{remark}

\subsection{Proof of~\cref{theorem 12}}
\label{contractibility smooth projective curve}

The goal of this section is to prove~\cref{theorem 12}, which states that the hybrid analytification of a smooth projective curve over certain subfields of $\mathbf{C}$ admits a strong deformation retraction onto the non-Archimedean fibre. 
The proof proceeds by reduction to~\cref{sec3 contractibility 2}, and the key topological tool do so is the~\cref{lifting_homotopy}.

\begin{proof}[Proof of~\cref{theorem 12}]
Let $k$ be a countable subfield of $\C$ that is not contained in $\R$.
Pick a finite ramified cover $\varphi \colon X \to \mathbf{P}^1_k$, write $Z \subseteq \P^{1}_k$ for the branch locus, and set $U \coloneqq \P^1_k \backslash Z$. 
It follows that $\varphi \colon \varphi^{-1}(U) \to U$ is a finite \'etale morphism, and hence $\varphi^{\hyb} \colon \varphi^{-1}(U)^{\hyb} \to U^{\hyb}$ is a local homeomorphism over the Archimedean fibres, by~\cref{local_inversion}.
In particular, if $C \coloneqq \P^{1,\mathrm{triv}}_k \cup Z^{\hyb}$ and $V \coloneqq \P^{1,\hyb}_k \backslash C$, then $\varphi^{\hyb}$ restricts to a local homeomorphism $(\varphi^{\hyb})^{-1}(V) \to V$. 

As explained in~\cref{contractibility for zariski opens}, there is a strong deformation retraction $H \colon [0,1] \times \P^{1,\hyb}_k \to \P^{1,\hyb}_k$ of $\P^{1,\hyb}_k$ onto the closed, contractible subset $C$.
Thus, \cref{lifting_homotopy} shows that there is a lift of $H$ to a strong deformation retraction of $X^{\hyb}$ onto $(\varphi^{\hyb})^{-1}(C) = X^{\triv} \cup R^{\hyb}$, where $R \subseteq X$ is the ramification locus of $\varphi$. 
The space $R^{\hyb}$ is a disjoint union of finitely-many intervals, each intersecting $X^{\mathrm{triv}}$ in one point; in particular, there is a strong deformation retraction of $(\varphi^{\hyb})^{-1}(C)$ onto $X^{\triv}$, as required.

Assume now that $k$ is contained in $\R$. 
Write $X^{\hyb}$ for the analytification of $X$, and $X_{k[i]}^{\hyb}$ for the analytification after ground field extension. As in the previous case, there is a finite branched cover $\varphi^{\hyb}\colon X^{\hyb}\to \P^{1,\hyb}_k$; further, 
there is commutative diagram 
\begin{center}
\begin{tikzcd}
X_{k[i]}^{\hyb} \arrow{r}{} \arrow[swap]{d}{\varphi^{\hyb}_{k[i]}} & X^{\hyb} \arrow{d}{\varphi^{\hyb}} \\
\P^{1,\hyb}_{k[i]} \arrow{r}{}& \P^{1,\hyb}_k
\end{tikzcd}
\end{center}
As in~\cref{sec3 contractibility 2}, we can pick
a strong retraction deformation $H$ of $\P^{1,\hyb}_{k[i]}$ onto $\P^{1,\triv}_{k[i]}$ such that for any $(t,x)\in [0,1]\times \P^{1,\hyb}_{k[i]}$, we have $I(H(t,x))=H(t,I(x))$. 
Following the argument of~\cref{corollary retraction disc}, the induced strong retraction deformation $\widetilde{H}$ of $X_{k[i]}^{\hyb}$ onto $X_{k[i]}^{\triv}$ satisfies the same property. 
In particular, $\widetilde{H}$ induces a strong retraction deformation $\widetilde{H}'$ of $X_k^{\hyb}$ onto $X_k^{\triv}$, as required.
\end{proof}

\subsection{Non-Contractibility of the Hybrid Affine Line over Uncountable Archimedean Fields}

\spa{
The goal of this section is to prove~\cref{not contractible} (stated as~\cref{theorem 1}(2) in the introduction), which asserts that $\A^{1,\hyb}_k$ is not contractible 
whenever $k$ is an uncountable subfield of $\C$ that is not contained in $\R$.
}

\begin{proposition}\label{non contractibility retraction}
For any $z \in k$, there is a retraction $\frakr_z \colon \A^{1,\hyb}_k \to \D_k(z)$ satisfying the following conditions:
\begin{enumerate}[label=(\roman*)]
\item $\lambda \circ \frakr_z = \lambda$;
\item if $x \in \A^{1,\triv}_k$ lies in a branch other than $[\eta_{z,0},\eta_{z,1})$, then $\frakr_z(x) = \eta_{z,1}$.
\end{enumerate}
In particular, $\frakr_z^{-1}(\eta_{z,r}) = \{ \eta_{z,r} \}$ for $r \in [0,1)$.
\end{proposition}

The retraction $\frakr_z$ is the hybrid analogue of the retraction of the complex plane $\A^{1,h}_{\C} = \C$ onto the closed disc of radius $1$ centered at $z$; in fact, $\frakr_z |_{\lambda^{-1}(1)}$ is precisely this map. 
On the non-Archimedean fibre, $\frakr_z$ contracts all branches of $\lambda^{-1}(0)$ to the trivial norm $\eta_{0,1}$ save for the branch of $z$, which is fixed.

Note that we do not claim that $\frakr_z$ extends to a strong deformation retraction of $\A^{1,\hyb}_k$ onto $\D_k(z)$ (even though this occurs on each Archimedean fibre). In fact, such an extension cannot always exist by~\cref{not contractible}.

\begin{proof}
Define the retraction $\frakr_z$ by the following formula:
$$
\frakr_z(x) \coloneqq \begin{cases}
\ev\left(\min\{ 1, t|w-z|_{\infty}^{-1}\} (w-z) + z, t \right) & x = \ev(w,t) \in \A^{1,\arch}_k, \\
 x & x \in [\eta_{z,0},\eta_{z,1}], \\
 \eta_{0,1} & \textrm{otherwise.}
\end{cases}
$$
It is easy to verify that $\frakr_{z}$ is the identity on $\D_k(z)$ and that it is continuous at any point of $\A^{1,\arch}_k \cup [\eta_{z,0},\eta_{z,1})$, since $\D_k(z)$ is a neighbourhood in $\A^{1,\hyb}_k$ of any point of $[\eta_{z,0},\eta_{z,1})$. 
Thus, it suffices to check the continuity of $\frakr_z$ at $x \in \A^{1,\triv}_k \backslash [\eta_{z,0},\eta_{z,1})$.
By~\cref{sec 3 nbhood basis}, we can assume that an open neighbourhood $V$ of $\eta_{0,1} = \frakr_z(x)$ in $\D_k(z)$ is of the form
$$
V = \{ y \in \A^{1,\hyb}_k \colon 1-\epsilon < |(T-z)(y)| < 1+\epsilon , \lambda(y) < \epsilon \} \cap \D_k(z)
$$
for some $\epsilon > 0$. 
In this case,
$$
\frakr_z^{-1}(V) = \left\{ y \in \A^{1,\hyb}_k \colon 1-\epsilon < \min\left\{ |(T-z)(y)|, \chi(\lambda(y)) \right\} < 1 + \epsilon , \lambda(y) < \epsilon\right\},
$$
which is clearly open and it contains all of $\A^{1,\mathrm{triv}}_k \backslash [\eta_{z,0},\eta_{z,1})$.
Thus, $\frakr_z$ is continuous everywhere.
\end{proof}

\spa{
The key tool in the proof~\cref{not contractible} is to prove the following:
if there exists a strong deformation retraction $H \colon [0,1] \times \A^{1,\hyb}_k \to \A^{1,\hyb}_k$ 
of $\A^{1,\hyb}_k$ onto the Gauss point $\eta_{0,1} \in \A^{1,\mathrm{triv}}_k$, then there must be a pair $(t,x) \in [0,1) \times \A^{1,\arch}_k$ such that $H(t,x)$ lies in the non-Archimedean fibre. 
The mechanism for proving such an assertion is the careful description of the topology of the hybrid closed disc $\D_k(z)$ in~\S\ref{local structure}
(and more precisely, \cref{corollary homotopy disc}): this is used below in~\cref{key lemma} and~\cref{non contractibility corollary 2} to impose restrictions on the behaviour of a homotopy between the inclusion $\D_k(z) \hookrightarrow \A^{1,\hyb}_k$ and a map $\D_k(z) \to \A^{1,\hyb}_k$ with image in the non-Archimedean fibre.
%
%
%
}

\begin{lemma}\label{key lemma}
Suppose $k$ is not contained in $\R$.
Let $z \in k$, and let $g \colon \D_k(z) \to \D_k(z)$ be a continuous map with image in $(\eta_{z,0},\eta_{z,1}]$.
If there is a homotopy $H \colon [0,1] \times \D_k(z) \to \D_k(z)$ from the identity map to $g$, then there exists $x \in \mathbf{E}_k(z)$ and $s \in (0,1]$ such that $H(s,x) = \eta_{z,0}$.
\end{lemma}

\begin{proof}
Suppose, for sake of contradiction, that for any $(s,x) \in (0,1] \times \mathbf{E}_k(z)$, we have $H(s,x) \not= \eta_{z,0}$.
Then, $H$ restricts to a homotopy 
$$
[0,1] \times \mathbf{E}_k(z) \to \D_k(z) \backslash \{ \eta_{z,0} \}
$$
between the inclusion $\mathbf{E}_k(z) \hookrightarrow \D_k(z) \backslash \{ \eta_{z,0} \}$ and $g |_{\mathbf{E}_k(z)}$.
The image of $g |_{\mathbf{E}_k(z)}$ is a connected subset of $(\eta_{z,0},\eta_{z,1})$ by assumption, hence $g |_{\mathbf{E}_k(z)}$ is homotopic to a constant map.
In particular, the inclusion $\mathbf{E}_k(z) \hookrightarrow \D_k(z) \backslash \{ \eta_{z,0} \}$ is homotopic to a constant map, which contradicts~\cref{corollary homotopy disc}.
\end{proof}

\begin{corollary}\label{non contractibility corollary 2}
Suppose $k$ is not contained in $\R$.
Let $z \in k$, and $g \colon \D_{k}(z) \to \A^{1,\hyb}_{k}$ be a continuous map with image in $\A^{1,\triv}_{k} \backslash \{ \eta_{z,0} \}$.
If $H \colon [0,1] \times \D_{k}(z) \to \A^{1,\An}_{k}$ is a homotopy from the inclusion $\D_{k}(z) \hookrightarrow \A^{1,\hyb}_{k}$ to $g$, then
 there exists $x \in \mathbf{E}_{k}(z)$ and $s \in (0,1]$ such that $H(s,x) = \eta_{z,0}$. 
\end{corollary}
%


\begin{proof}
Apply~\cref{key lemma} to the homotopy $\frakr_z \circ H$, where $\frakr_z \colon \A^{1,\hyb}_k \to \D_k(z)$ is the retraction of~\cref{non contractibility retraction}.
\end{proof}

\begin{lemma}\label{fixed point lemma}
If $\varphi \colon \A^{1,\hyb}_{k} \to \A^{1,\hyb}_{k}$ is a continuous map that does not fix $\eta_{0,1} \in \A^{1,\hyb}_{k}$, then the image $\varphi(\A^{1,\hyb}_{k})$ intersects at most countably-many branches of $\A^{1,\hyb}_{k}$. 
\end{lemma}


\begin{proof}
We begin with two preliminary observations.
\begin{enumerate}
%
\item The complement of a closed neighbourhood of $\eta_{0,1}$ is second countable. Indeed, any such open set is covered by the union of $\A^{1,\arch}_k$ and the interiors of finitely-many hybrid discs $\D_k(p)$ (namely, those that correspond to branches not contained in the neighbourhood of $\eta_{0,1}$).
Each member of this finite cover is second countable, and hence the open set is second countable.

\item Any point of $\A^{1,\hyb}_k \backslash \{ \eta_{0,1}\}$ has an open neighbourhood that intersects at most one branch. 
Indeed, if the point lies in $\A^{1,\arch}_k$, then take a neighbourhood contained entirely in $\A^{1,\arch}_k$.
On the other hand, if the point lies in $\A^{1,\triv}_k$, it is of the form $\eta_{p,r}$ for some monic, irreducible polynomial $p \in k[T]$ and $r \in [0,1)$, and $\D_k(p)$ is a neighbourhood of $\eta_{p,r}$ that only intersects the branch $[\eta_{p,0},\eta_{p,1})$.
%
\end{enumerate}
As $\varphi(\eta_{0,1}) \not= \eta_{0,1}$, the observation (2) asserts that there is an open neighbourhood $U$ of $\varphi(\eta_{0,1})$ that intersects at most one branch.
The space $\A^{1,\hyb}_{k}$ is locally compact, so there is a compact neighbourhood $V \subseteq U$ of $\varphi(\eta_{0,1})$; in particular, $V$ intersects at most one branch.
By the observation (1), 
the complement $W \coloneqq \A^{1,\hyb}_{k} \backslash \varphi^{-1}(V)$ is second countable.
Thus,
$\varphi^{-1}(\varphi(W) \cap \A^{1,\triv}_k)$ is a subspace of $\varphi^{-1}(\varphi(W)) = W$, and hence it is also second countable.
%
%
%
%
%
%
%
%
%
In particular, $\varphi(W) \cap \A^{1,\triv}_k$ 
is a topological space with at most countably-many pairwise-disjoint open sets;
thus, $\varphi(W)$ intersects at most countably-many branches.
It follows that
$$
\varphi(\A^{1,\hyb}_{k}) = \varphi(W) \cup \varphi(\varphi^{-1}(V)) \subseteq \varphi(W) \cup V
$$
intersects at most countably-many branches, since $\varphi(W)$ intersects at most countably-many branches and $V$ intersects at most one branch.
\end{proof}

\begin{lemma}\label{old claim}
Let $k $ be an uncountable subfield of $\C$ that is not contained in $\R$.
If there exists a homotopy $H \colon [0,1] \times \A^{1,\hyb}_{k} \to \A^{1,\hyb}_{k}$ between the identity map on $\A^{1,\hyb}_{k}$ and the constant map with value $\eta_{0,1} \in \A^{1,\triv}_k$, then the image $H([0,1] \times \A^{1,\arch}_k)$ of the Archimedean fibres intersects uncountably-many branches.
\end{lemma}

\begin{proof}
Consider the subset
$$
F \coloneqq \{ s \in [0,1] \colon H(s,\A^{1,\hyb}_{k}) \textrm{ intersects at most countably-many branches of $\A^{1,\triv}_{k}$} \}.
$$
It is non-empty: indeed, $1 \in F$ since $H(1,\A^{1,\hyb}_k) = \{ \eta_{0,1} \}$ intersects no branches. 
Further, $0 \not\in F$ since $H(0,\A^{1,\hyb}_{k}) = \A^{1,\hyb}_{k}$ intersects all branches of $\A^{1,\hyb}_{k}$, and there are uncountably-many by the assumption on $k$.
Let $s_0 \coloneqq \inf\{ s \colon s \in F\}$, and note that $s_0 > 0$.
\begin{enumerate}
\item[Case 1.]
Suppose $s_0 \in F$.
Let $B_0$ be the (necessarily uncountable) set of branches of the form $[\eta_{z,0},\eta_{z,1})$ of $\A^{1,\triv}_k$ 
such that 
$$
H(s_0,\A^{1,\hyb}_{k}) \cap [\eta_{z,0},\eta_{z,1}) = \emptyset.
$$
Fix one branch
$[\eta_{z,0},\eta_{z,1})$ in $B$,
and it suffices to show that $H([0,1] \times \A^{1,\arch}_k)$ intersects $[\eta_{z,0},\eta_{z,1})$.

\cref{non contractibility corollary 2} implies that there exists $x_z \in \mathbf{E}_{k}(z)$ and $s_1 \in [0,s_0)$ such that $H(s_1,x_z) = \eta_{z,0}$.
If $x_z \in \A^{1,\arch}_k$, then we are done.
Assume, for sake of contradiction, that $x_z \in \A^{1,\triv}_k$.
We know that $\mathbf{E}_{k}(z) \cap \A^{1,\triv}_k = \{ \eta_{0,1}\}$ (see~\cref{hybrid disc topology}),
hence $x_z = \eta_{0,1}$. 
By the hypothesis on $s_0$, $H(s_1,\A^{1,\hyb}_{k})$ intersects uncountably-many branches, and hence the map $H(s_1,\cdot) \colon \A^{1,\hyb}_{k} \to \A^{1,\hyb}_{k}$ must fix $\eta_{0,1}$ by~\cref{fixed point lemma}; however, we know that $H(s_1,x_z) = \eta_{z,0}$, a contradiction. 

\item[Case 2.] 
If $s_0 \not\in F$, then $s_0 < 1$.
Let $\overline{B}$ denote the (uncountable) set of branches of $\A^{1,\hyb}_{k}$.
The definition of $s_0$ guarantees that for any $n \geq 1$, there exists $s_n \in (s_0,\min\{ 1,s_0 + \frac{1}{n}\})$ such that $H(s_n,\A^{1,\hyb}_{k})$ intersects at most countably-many branches; said differently, the set $B_n \subseteq \overline{B}$ of branches that \emph{do not} intersect $H(s_n,\A^{1,\hyb}_{k})$ is uncountable.
It follows that $B \coloneqq \bigcap_{n\geq 1} B_n$ is uncountable, since its complement in $\overline{B}$ is countable: indeed,
$$
\overline{B} \backslash B = \bigcup_{n \geq 1} \overline{B}\backslash B_n,
$$
and each $\overline{B}\backslash B_n$ consists of the (at most countable) collection of branches that intersect $H(s_n,\A^{1,\hyb}_{k})$.

Now, consider a branch of the form $[\eta_{z,0},\eta_{z,1})$ that lies in $B$.
\cref{non contractibility corollary 2} guarantees that for all $n \geq 1$, there exists $s_n' \in [0,s_n)$ and $x_n \in \mathbf{E}_{k}(z)$ such that $H(s_n',x_n) = \eta_{z,0}$.
The space $\mathbf{E}_{k}(z) = \mathbf{E}_{k}(z,1)$ is sequentially compact,
hence the product $\mathbf{E}_{k}(z) \times [0,1]$ is as well; thus, after passing to a subsequence, we may assume that the sequence $\{ (s_n',x_n)\}_{n \geq 1} \subseteq [0,1] \times \mathbf{E}_{k}(z)$ converges to a point $(s,x) \in [0,s_0] \times \mathbf{E}_{k}(z)$. 
Moreover, by the continuity of $H$, we have $H(s,x) = \eta_{z,0}$. 
Arguing as in Case 1, it follows that $x \in \A^{1,\arch}_k$, as required.
\end{enumerate}
\end{proof}

\begin{theorem}\label{not contractible}
If $k$ is an uncountable subfield of $\C$ that is not contained in $\R$, then $\A^{1,\hyb}_{k}$ is not contractible. 
\end{theorem}


\begin{proof}
The space $\A^{1,\hyb}_k$ is path-connected by~\cref{prop:path-connected},
so assume for sake of contradiction that there is a homotopy 
$$
H \colon [0,1] \times \A^{1,\hyb}_{k} \to \A^{1,\hyb}_{k}
$$
between the identity map on $\A^{1,\hyb}_{k}$ and the constant map with value $\eta_{0,1} \in \A^{1,\triv}_k$. %
By~\cref{old claim}, there is an uncountable set $B$ of branches that $Z \coloneqq H( [0,1] \times \A^{1,\arch}_k)$ intersects. 
%
If the branch $[\eta_{p,0},\eta_{p,1})$ lies in $B$, then $[\eta_{p,0},\eta_{p,1}) \cap Z$ is non-empty, and it is open in $\A^{1,\triv}_k \cap Z$. 
In particular, $\A^{1,\triv}_k \cap Z$ contains uncountably-many disjoint open subsets.
If $\widetilde{H} \colon [0,1] \times \A^{1,\arch}_k \to \A^{1,\hyb}_k$ denotes the restriction of $H$ to $[0,1] \times \A^{1,\arch}_k$, then 
$\widetilde{H}^{-1}(\A^{1,\triv}_k \cap Z)$ does as well.
However, $\widetilde{H}^{-1}(\A^{1,\triv}_k \cap Z)$ is a subspace of $[0,1] \times \A^{1,\arch}_k$, which is 
second countable (indeed, it is homeomorphic to the product of $[0,1] \times (0,1]$ with the complex plane).
This yields a contradiction.
\end{proof}


\section{Contractibility of affine space over non-Archimedean hybrid fields and dvrs}

\spa{
The goal of this section is to establish the homotopy equivalence of~\cref{theorem 2} over both a non-Archimedean field equipped with the hybrid norm, and a dvr equipped with the trivial norm. 
The latter case is used in \S\ref{section contractibility integers} to prove the analogous result over a Dedekind domain equipped with the trivial norm, which is an important component to the proof of~\cref{theorem 3}.
%
}

\spa{Throughout this section, we will treat the field and dvr cases simultaneously.
To this end, 
$A$ will denote either a field that is complete with respect to a non-Archimedean absolute value $|\cdot |_{\mathrm{field}}$, or a discrete valuation ring. 
If $A$ is a non-Archimedean field, write $| \cdot |_A$ for the hybrid norm $\max\{ |\cdot|_0, | \cdot |_{\mathrm{field}}\}$; 
if $A$ is a dvr, write $| \cdot |_A$ for the trivial norm on $A$, and $| \cdot |_{\mathrm{dvr}}$ for a norm given by the discrete valuation (and normalized in any way). 
In both cases, the pair $(A, |\cdot |_A)$ is a Banach ring (and, in fact, it is a geometric base ring).
}

\spa{Consider the homeomorphism $\alpha \colon [0,1] \stackrel{\simeq}{\to} \M(A,| \cdot |_A)$ given by the following rules:
\begin{itemize}
\item if $A$ is a non-Archimedean field, set 
$$
\alpha(\rho) \coloneqq |\cdot |_{\mathrm{field}}^{\rho},
$$
where $\alpha(0)$ is understood to be trivial norm on $A$;
\item if $A$ is a dvr, set 
$$
\alpha(\rho) \coloneqq |\cdot |_{\mathrm{dvr}}^{-\log(1-\rho)},
$$
where $\alpha(0)$ is understood to be the trivial norm, and $\alpha(1)$ is the seminorm given by reducing mod the maximal ideal and applying the trivial norm.
\end{itemize}
In both cases, note that $\alpha(0)$ is the trivial norm on $A$. 
As before, we omit the dependence on $| \cdot |_A$ from the notation for the spectrum, and simply write $\mathcal{M}(A)$. 
}

\spa{
Let $X = \A^{n}_A = \Spec(B)$ be the affine $n$-space over $A$ 
with coordinates $T_1,\ldots,T_n$, 
where $B = A[T_1,\ldots,T_n]$.
Write $\lambda \colon X^{\An} \to \M(A)$ for the analytification, and $X^{\triv} \coloneqq (\alpha^{-1} \circ \lambda)^{-1}(0) \subseteq X^{\An}$ for the fibre above $0$.
Note that if $A$ is a field, then $X^{\triv}$ is the trivially-valued analytification of $X$; if $A$ is a dvr, then $X^{\triv}$ is the trivially-valued analytification of the base change $\A^n_{\mathrm{Frac}(A)}$.
%
}

\subsection{The Hybrid Toric Skeleton}\label{section_41}

\spa{
Consider the $A$-torus $T = \Spec(A [u_1^{\pm 1},\ldots, u_n^{\pm 1}])$ of dimension $n$, which acts on $X$ by the multiplication map $\mathrm{m} \colon T \times_A X \to X$; that is, $\mathrm{m}$ is induced by the $A$-algebra homomorphism $\mathrm{m}^* \colon B \to B[u_1^{\pm 1},\ldots,u_n^{\pm 1}]$ given by 
$$
T_i \mapsto \mathrm{m}^*(T_i) \coloneqq u_i T_i.
$$
Further, consider the $B$-algebra homomorphism $F \colon B[u_1^{\pm 1},\ldots,u_n^{\pm 1}] \to B[[v_1,\ldots,v_n]]$ given by
$$
u_i \mapsto F(u_i) \coloneqq 1 + v_i.
$$
Said differently, if $f \in B[u_1^{\pm 1},\ldots,u_n^{\pm 1}]$, then $F(f)$ is the Taylor series expansion of $f$ about $(1,...,1)$ in $T \times_A X$, the $n$-dimensional torus over $X$. 
There is a unique way to write the image $F(f)$ in the form 
$$
F(f) = \sum_{\mu \in (\Z_{\geq 0})^n} F(f)_{\mu} v^{\mu},
$$
where $F(f)_{\mu} \in B$. 
These coefficients $F(f)_{\mu}$ can be constructed explicitly: if $f = \sum_{\nu \in (\Z_{\geq 0})^n} b_{\nu} u^{\nu}$ for $b_{\nu} \in B$, then the binomial formula shows that
\begin{equation}\label{formula for coefficients}
F(f)_{\mu} = \sum_{\nu \in M(f)} {\nu \choose \mu} b_{\nu},
\end{equation}
where $M(f) \subseteq (\Z_{\geq 0})^n$ is the finite subset of multi-indices $\nu$ such that $b_{\nu} \not= 0$.
In particular, $F(f)_{\mu} = 0$ for $|\mu| \gg 0$.
}

\spa{\label{affine toric lemma 1}
For $t = (t_1,\ldots,t_n) \in [0,1]^n$ and $x \in X^{\An}$, 
we will construct a point
$p(t,x) \in (X \times_A T)^{\An}$ 
that maps to $x$ under the analytification $(T \times_A X)^{\An} \to X^{\An}$ of the projection map.
It is defined as follows:
if $f \in B[u_1^{\pm 1},\ldots,u_n^{\pm 1}]$, then
\begin{equation}\label{formula for p}
|f(p(t,x))| \coloneqq \sup_{\mu \in (\Z_{\geq 0})^n} |F(f)_{\mu}(x)| t^{\mu}.
\end{equation}
As $F(f)_{\mu} = 0$ for $|\mu | \gg 0$, this supremum is in fact a maximum.
%
}

\spa{The map $p \colon [0,1]^n \times X^{\An} \to (T \times_A X)^{\An}$, given by $(t,x) \mapsto p(t,x)$, is continuous, and it is used below to construct a strong deformation retraction of $X^{\An}$ onto $X^{\triv}$.
Let $\ell \colon [0,1] \hookrightarrow [0,1]^n$ denote the diagonal map given by $t \mapsto (t,\ldots,t)$.
}

\begin{proposition}\label{affine toric lemma 2}
The map $q \colon [0,1] \times X^{\An} \to X^{\An}$, given by the composition
$$
[0,1] \times X^{\An} \stackrel{(\ell,\mathrm{id})}{\hookrightarrow} [0,1]^n \times X^{\An} \stackrel{p}{\longrightarrow} (T \times_A X)^{\An} \stackrel{\mathrm{m}^{\An}}{\longrightarrow} X^{\An},
$$
is continuous.
\end{proposition}

The map $q$ is the hybrid analogue of the map constructed by Thuillier in~\cite[Proposition 2.17]{thuillier}, and the proof of continuity follows the same argument. 
In order to simplify the proof of~\cref{affine toric lemma 2}, as well as later results, we first give an explicit formula for the point $q(t,x)$, which is analogous to the formulas appearing~\cite[\S 6.1]{berkovich} 
(indeed, the expression inside the maximum below in~\cref{affine toric eqn 9} is precisely the ``derivative'' $\del_{\nu}$ that appears in~\cite[Remark 6.1.3(ii)]{berkovich}). 

\begin{lemma}\label{affine toric lemma 4}
For any $t \in [0,1]$, $x \in X^{\An}$, and $f = \sum_{\mu \in (\Z_{\geq 0})^n} a_{\mu} T^{\mu} \in B$, we have
\begin{equation}\label{affine toric eqn 9}
|f(q(t,x))| = \max_{\nu \in (\Z_{\geq 0})^n} 
\left| \sum_{\mu \in (\Z_{\geq 0})^n} {\mu \choose \nu} a_{\mu} T^{\mu} (x)\right| t^{|\nu|}.
\end{equation}
Furthermore, we have
\begin{equation}\label{affine toric eqn 13}
|f(q(1,x))| = \max_{\mu \in (\Z_{\geq 0})^n} |a_{\mu}(x)| \cdot |T^{\mu}(x)|.
\end{equation}
\end{lemma}

\begin{proof}
We can write $\mathrm{m}^*(f) = \sum_{\mu \in (\Z_{\geq 0})^n} a_{\mu} T^{\mu} u^{\mu}$ in $B[u_1^{\pm 1},\ldots,u_n^{\pm 1}]$, and 
expanding $F(\mathrm{m}^*(f))$ as in~\cref{formula for coefficients} yields the equality
$$
F(\mathrm{m}^*(f)) = 
\sum_{\nu \in (\Z_{\geq 0})^n} \left( \sum_{\mu \in (\Z_{\geq 0})^n} {\mu \choose \nu} a_{\mu} T^{\mu} \right) v^{\nu},
$$
in $B[[v_1,\ldots,v_n]]$. 
The equation~\cref{affine toric eqn 9} follows by applying the seminorm $p(t,x)$ to the above expression.

As $| {\mu \choose \nu} (x)| \leq 1$ for all multi-indices $\mu,\nu$, it follows immediately from~\cref{affine toric eqn 9} that $|f(q(1,x))|$ is bounded above by the right-hand side of~\cref{affine toric eqn 13}. 
Assume that the right-hand side of~\cref{affine toric eqn 13} is strictly positive, otherwise there is nothing to prove.
Pick $\mu_0 \in (\Z_{\geq 0})^n$ among those multi-indices $\mu$ with $|a_{\mu}(x)| \cdot |T^{\mu}(x)|$ maximal, such that $|\mu_0|$ is maximal.
Let $J \subseteq (\Z_{\geq 0})^n$ be the subset of multi-indices $\mu$ with $|a_{\mu}(x)| \cdot |T^{\mu}(x)|$ maximal, and pick $\mu_0 \in J$ such that $|\mu_0|$ is maximal. Note that if $n > 1$, then $\mu_0$ is not necessarily unique. 

By taking $\nu = \mu_0$ in~\cref{affine toric eqn 9}, we see that 
\begin{equation}\label{lemma temp eqn}
|f(q(1,x))| \geq \left| \left(a_{\mu_0} T^{\mu_0} + \sum_{\mu \not= \mu_0} {\mu \choose \mu_0} a_{\mu} T^{\mu}\right)(x) \right|. 
\end{equation}
Suppose there is a multi-index $\mu \not= \mu_0$ such that $|a_{\mu_0}(x) | \cdot |T^{\mu_0}(x)| = | {\mu \choose \mu_0} a_{\mu} T^{\mu}(x)|$, in which case $\mu \in J$ and $|\mu| = |\mu_0|$. However, in order for ${\mu \choose \mu_0}$ to be nonzero, we must have that $\mu_i \geq (\mu_0)_i$ for all $i =1,\ldots,n$; this cannot occur if $\mu \not= \mu_0$ and $|\mu| = |\mu_0|$.
Thus, $|a_{\mu_0}(x)| \cdot |T^{\mu_0}(x)| > |{\mu \choose \mu_0} a_{\mu} T^{\mu}(x)|$ for all $\mu \not= \mu_0$, and hence the right-hand side of~\cref{lemma temp eqn} is equal to $|a_{\mu_0}(x)| \cdot |T^{\mu_0}(x)|$, as required.
\end{proof}


\begin{proof}[Proof of~\cref{affine toric lemma 2}]
The strategy of proof follows that of~\cite[Proposition 2.17(i)]{thuillier}.
Fix $f = \sum_{\mu} a_{\mu} T^{\mu}$ in $B$, and set
$\varphi(t,x) \coloneqq |f(q(t,x))|$. 
It suffices to show that $\varphi$ is continuous as a map $[0,1] \times X^{\An} \to \R$. 
By~\cref{affine toric lemma 4}, it is 
easy to see that $\varphi$ is monotonic in $t$, i.e.\ $\varphi(t,x) \leq \varphi(t',x)$ provided $t \leq t'$. 
Thus, by arguing as in~\cite[Proposition 2.17(i)]{thuillier}, it suffices to prove that $\varphi$ is continuous in each variable separately.

For $x \in X^{\An}$ fixed, we will first show that $\varphi(\cdot,x)$ is continuous. 
Pick $N > 0$ such that $F(f)_{\mu} = 0$ for all $|\mu| > N$, in which case
$$
\varphi(t,x) = \max_{|\nu| \leq N} |F(\mathrm{m}^*(f))_{\nu}(x)| t^{|\nu|}
$$
is a maximum of finitely-many continuous functions in $t$, hence it is itself continuous.

For fixed $t \in [0,1]$, we will show that $\varphi(t,\cdot)$ is continuous.
Write $\mathrm{m}^*(f) = \sum_{\mu \in M(f)} a_{\mu} T^{\mu} u^{\mu}$, where recall that $M(f) \subseteq (\Z_{\geq 0})^n$ is the finite subset of multi-indices $\mu$ such that $a_{\mu} \not= 0$. 
If $t= 0$, then 
$$
\varphi(0,x) = \left| \sum_{\mu \in M(f)} a_{\mu} (x) \right|,
$$
which is clearly continuous in $x$. If $t = 1$, 
then~\cref{affine toric lemma 4} shows that
$$
|f(q(1,x))| = \max_{\mu \in M(f)} |a_{\mu}(x)| \cdot |T^{\mu}(x)|,
$$
In particular, $\varphi(1,\cdot)$ is a maximum of finitely-many continuous functions, hence it is itself continuous. 
 
Assume now that $t \in (0,1)$. The function $\varphi(t,\cdot)$ is the maximum of a collection of continuous functions, and hence it is lower-semicontinuous. 
It remains to check that it is upper-semicontinuous.
For each $R > 0$, consider the open set
$$
U_R \coloneqq \left\{ x \in X^{\An} \colon \max_{\mu \in M(f)} |a_{\mu}(x)| < R \right\}.
$$
It suffices to verify that $\varphi(t,\cdot)$ is upper-semicontinuous after restricting to each member of the open cover $\{ U_R \}_{R > 0}$ of $X^{\An}$.
Fix $R > 0$ and $x_0 \in U_R$. 
For any $\epsilon > 0$, there exists $n_{\epsilon} > 0$ such that 
$$
R t^{|\nu|} < \varphi(t,x_0) + \epsilon
$$
for all multi-indices $\nu$ with $|\nu| > n_{\epsilon}$. 
Furthermore, 
$|F(\mathrm{m}^*(f))_{\nu}(x)| t^{|\nu|} \leq Rt^{|\nu|}$ for all $\nu$, so
it follows that
$$
\left\{ x \in U_R \colon \varphi(t,x) < \varphi(t,x_0) + \epsilon \right\} = \bigcap_{|\nu| \leq n_{\epsilon}} \left\{ x \in U_R \colon |F(\mathrm{m}^*(f))_{\nu}(x)| t^{|\nu|} < \varphi(t,x_0) + \epsilon  \right\}
$$
is an open neighbourhood of $x_0$.
Thus, $\varphi(t,\cdot)$ is upper-semicontinuous at $x_0$,
which
completes the proof.
\end{proof}

\begin{remark}
In order to construct the map $p$ (and, by extension, the map $q$) as in~\cite{berkovich,thuillier}, one would require a robust theory of relative Shilov boundaries (in particular, one that is compatible with base change).
If such a theory were to exist, one could proceed as follows: for each $t \in [0,1]$, construct a section $g_t$ of $T^{\An} \to \M(A)$ and a ``completed residue $A$-algebra'' $\H(g_t)$ of $g_t$ such that the point $p(t,x)$ is the unique Shilov point of $\M(\H(x)) \times_{\M(A)} \M(\H(g_t))$ for $x \in X^{\An}$. 
In order to make the above precise, one must make sense of $\H(g_t)$, and likely prove a version of~\cite[\S 3, Thm 1(4)]{gruson} over more general Banach rings. 
%
%
%
\end{remark}


\spa{
The \emph{hybrid toric skeleton of $X$} is the subset  $Z_X \subseteq X^{\An}$ consisting of those points $x \in X^{\An}$ such that
$$
|f(x)| = \max_{\mu \in (\Z_{\geq 0})^n} |a_{\mu}(x)| \cdot |T^{\mu}(x)| 
$$
for any $f = \sum_{\mu \in (\Z_{\geq 0})^n} a_{\mu} T^{\mu}$ in $B$. 
For any $z \in \M(A)$, the intersection $Z_X \cap \lambda^{-1}(z)$ coincides with the toric skeleton of $\lambda^{-1}(z)$ in the sense of~\cite{thuillier}, after identifying $\lambda^{-1}(z)$ with the analytification of $\A^{n}_{\H(z)}$.
}

\begin{proposition}\label{affine toric prop 3}
The map $q \colon [0,1] \times X^{\An} \to X^{\An}$ 
is a strong deformation retraction of $X^{\An}$ onto $Z_X$, and it satisfies $\lambda(q(t,x)) = \lambda(x)$ for any $(t,x) \in [0,1] \times X^{\An}$.
\end{proposition}

\begin{proof}[Proof of~\cref{affine toric prop 3}]
%
%
It follows immediately from~\cref{affine toric lemma 4} that $q(0,x) =x$ and $q(1,x) \in Z_X$.
It remains to show that $q(t,x) \in Z_X$ for $x \in Z_X$ and $t \in [0,1]$. 
If $f  \in B$, then~\cref{affine toric lemma 4} shows that $|f(q(t,x))| \geq |f(x)|$, and hence 
$$
|f(x)| \leq |f(q(t,x))| \leq |f(q(1,x))| = |f(x)|,
$$
where the final equality follows from~\cref{affine toric lemma 4} and the assumption that $x \in Z_X$.
\end{proof}

\subsection{Proof of~\cref{theorem 2}}\label{section_42}

\spa{
For $t \in [0,1]$ and $x \in X^{\An}$, set $\beta(t,x) \coloneqq \min\{ 1-t,(\alpha^{-1} \circ \lambda)(x) \}$; this defines a continuous function $\beta \colon [0,1] \times X^{\An} \to [0,1]$. 
Further, define a point $J(t,x) \in X^{\An}$ by the formula
\begin{equation}\label{affine toric eqn 6}
f = \sum_{\mu \in (\Z_{\geq 0})^n} a_{\mu} T^{\mu} \mapsto |f(J(x,t))| \coloneqq \max_{\mu \in (\Z_{\geq 0})^n} |a_{\mu}( (\alpha \circ \beta)(t,x))| \cdot |T^{\mu}(x)|.
\end{equation}
The function $J$ fixes a point $x \in Z_X$ whenever $(\alpha^{-1} \circ \lambda)(x) > 1-t$, and sends $x$ to a point of $Z_X \cap (\alpha^{-1} \circ \lambda)^{-1}(1-t)$ otherwise. These properties are made precise in the sequence of lemmas below.
}

\begin{lemma}\label{affine toric lemma 5}
The map $J \colon [0,1] \times X^{\An} \to X^{\An}$ is continuous, it restricts to a map $[0,1] \times Z_X \to Z_X$, and it satisfies 
\begin{equation}\label{affine toric eqn 7}
\lambda(J(t,x))) = \alpha(\beta(t,x))
\end{equation}
for all $(t,x) \in [0,1] \times X^{\An}$.
\end{lemma}

\begin{proof}
The formula~\cref{affine toric eqn 7} is clear, as is the assertion that $J$ restricts to a map $[0,1] \times Z_X \to Z_X$.
It remains to show that $J$ is continuous.
Fix $f = \sum_{\mu} a_{\mu} T^{\mu}$ and $f' = \sum_{\mu} a_{\mu}' T^{\mu}$ in $B$ and $r, r' > 0$. 
It suffices to show that the set $J^{-1}(U)$ is open, where $U=\{ x \in X^{\An} \colon |f(x)| < r , |f'(x)| > r' \}$. 
We can write
$$
J^{-1}(U) = \left( \bigcap_{\mu \colon a_{\mu} \not=0} 
\{ (t,x) \colon |a_{\mu}( (\alpha \circ \beta)(t,x))|\cdot |T^{\mu}(x)| < r \}
\right) \bigcap \left( \bigcup_{\mu}
\{ (t,x) \colon |a'_{\mu}( (\alpha \circ \beta)(t,x))| \cdot |T^{\mu}(x)| > r' \}
\right)
$$
Write $V(a_{\mu},r)$ for a subset in the left-hand intersection, and write $W(a_{\mu}',r')$ for a subset in the right-hand intersection.
It suffices to show that $V(a_{\mu},r)$ and $W(a_{\mu}',r')$ are open. By expanding the definition of $\beta(t,x)$ and rearranging, we have
\begin{align*}
V(a_{\mu},r)
&= \bigcup_{\delta \in \Q_{>0}} \bigg( \{ (t,x) \colon |a_{\mu}( (\alpha \circ \beta)(t,x))| < \delta \} \cup \left\{ (t,x) \colon \delta |T^{\mu}(x)| < r \right\} \bigg). 
\end{align*}
For $\delta \in \Q_{>0}$, the right-hand set in the union is clearly open.
The left-hand set is open since it is the preimage of $\{ y \in \M(A) \colon |a_{\mu}(y)| < \delta \}$ under the continuous map $\alpha \circ \mu$.
Thus, $V(a_{\mu},r)$ is open.
Similarly, one shows that $W(a_{\mu}',r')$ is open. 
\end{proof}

\spa{
Consider the image $Z_{X,0} \coloneqq J(1,Z_X) \subseteq Z_X$ of the hybrid toric skeleton under $J(1,\cdot)$. 
As $\beta(1,x) = 0$ for all $x \in X^{\An}$, it follows from~\cref{affine toric eqn 7} that $Z_{X,0} \subseteq X^{\triv}$. 
In fact, $Z_{X,0}$ coincides with the toric skeleton of $X^{\triv}$ in the sense of~\cite{thuillier}.
}

\begin{lemma}\label{affine toric lemma 6}
The restriction of $J$ to $Z_X$ is a strong deformation retraction of $Z_X$ onto $Z_{X,0}$.
\end{lemma}

\begin{proof}
It is clear from the definition that $J(0,x) = x$ for $x \in Z_X$ (since $\beta(x,0) = \alpha(\lambda(x))$), and $J(1,x) \in Z_{X,0}$ by the definition of $Z_{X,0}$.
Moreover, for any $t \in [0,1]$, $\beta(t,x) = 0 = \alpha^{-1}(\lambda(x))$ whenever $x \in Z_{X,0}$, so $J(t,x) = x$ for such $x$.
\end{proof}

\begin{lemma}\label{affine toric lemma 7}
The restriction of $q$ to 
$X^{\triv}$
$(\alpha^{-1} \circ \lambda)^{-1}(0)$ 
is a strong deformation of 
$X^{\triv}$
$(\alpha^{-1} \circ \lambda)^{-1}(0)$ 
onto $Z_{X,0}$.
\end{lemma}

\begin{proof}
This is the assertion~\cite[Proposition 2.17]{thuillier}, after replacing the $\beth$-analytification $X^{\beth}$ with $X^{\triv}$ (though the argument is identical). 
Alternatively, it is straightforward to verify the assertion directly.
\end{proof}

\begin{proof}[Proof of~\cref{theorem 2}]
Consider the commutative diagram of inclusions
\begin{center}
\begin{tikzcd}
Z_X \arrow[hook]{r}{} & X^{\An} \\
Z_{X,0} \arrow[hook]{u}{} \arrow[hook]{r}{} & X^{\triv} \arrow[hook]{u}{}
\end{tikzcd}
\end{center}
The inclusion $Z_X \hookrightarrow X^{\An}$ is a homotopy equivalence by~\cref{affine toric lemma 2}, the inclusion $Z_{X,0} \hookrightarrow Z_X$ is a homotopy equivalence by~\cref{affine toric lemma 6}, and the inclusion $Z_{X,0} \hookrightarrow X^{\triv}$ is a homotopy equivalence by~\cref{affine toric lemma 7}.
It follows that $X^{\triv} \hookrightarrow X^{\An}$ is a homotopy equivalence. 
\end{proof}

\section{Contractibility of the affine line over the ring of integers of a number field}\label{section contractibility integers}
In this section, we discuss an extension of~\cref{theorem 2} to an affine spaces over Dedekind domains, which we then combine with~\cref{theorem 1} to prove~\cref{theorem 3}.

\subsection{Proof of~\cref{theorem 2} over a Dedekind Domain}\label{extension_to_dedekind_domains}
Let $A$ be a Dedekind domain equipped with the trivial norm $| \cdot |_0$. 
While $\mathcal{M}(A)$ is no longer homeomorphic to $[0,1]$, it admits a simple (set-theoretic) description as the disjoint union of $\mathcal{M}(A_{\frakp},|\cdot |_0)$ over $\frakp \in \Spec(A)$, but identify the trivial norm of the spectra of the discrete valuation rings $A_{\frakp}$'s; see~\cite[1.4.2]{berkovich} or~\cref{fig:dedekind}.
For this reason, the proof of~\cref{theorem 2} can be carried through in this setting with a few small modifications, which we illustrate below.

For each $\frakp \in \Spec(A)$, pick a uniformizer $\pi_{\frakp}$ of $A_{\frakp}$.
For any $z \in \M(A)$, set
$$
\mathrm{ht}(z) \coloneqq \prod_{\frakp \in \Spec(A)} |\pi_{\frakp}(z)|.
$$
For each $z$, all but one factor of the above product is equal to $1$, because $|\pi_{\frakp}(z)| = 1$ whenever $\frakp \not= \mathrm{red}(z)$, where $\mathrm{red} \colon \M(A) \to \Spec(A)$ denotes the reduction map. 
Moreover, $\mathrm{ht}$ is independent of the choice of uniformizers $\pi_{\frakp}$ because any two differ by a unit $u$ in $A_{\frakp}$, and $|u(z)| = 1$.
It is easy to check that $\mathrm{ht}$ is continuous. 

\begin{figure}[h]
  \centering
  \begin{tikzpicture}
    \draw[thick] (1,0) -- (1,2);
    \node[anchor=south] at (1,2) {$| \cdot |_0$};
    \node at (1,2) {$\bullet$};
    \node at (1,1) {$\bullet$};
    \node[anchor=west] at (1,1) {$|\cdot |_{\frakp,c}$};
    \node at (1,0) {$\bullet$};
    \node[anchor=north] at (1,0) {$| \cdot |_{\frakp,0}$};
    \draw[thick] (-1,0) -- (1,2);
    \node at (-1,0) {$\bullet$};
    \node[anchor=north] at (-1,0) {$|\cdot |_{\frakp',0}$};
    \node at (0,0) {$\ldots$};
    \node at (-0.5,0.5) {$\bullet$};
    \node[anchor=east] at (-0.5,0.5) {$|\cdot |_{\frakp',c'}$};
    \draw[thick] (3,0) -- (1,2);
    \node at (3,0) {$\bullet$};
    \node[anchor=north] at (3,0) {$|\cdot |_{\frakp'',0}$};
    \node at (2,0) {$\ldots$};
    \node at (1.5,1.5) {$\bullet$};
    \node[anchor=west] at (1.5,1.5) {$|\cdot |_{\frakp'',c''}$};
    \draw (5,0) -- (5,2);
    \node[anchor=north] at (5,0) {$\mathrm{ht}$};
    \node at (5,0) {$-$};
    \node[anchor=west] at (5,0) {$0$};
    \node at (5,2) {$-$};
    \node[anchor=west] at (5,2) {$1$};
    \node at (5,1) {$-$};
    \node[anchor=west] at (5,1) {$c$};
    \node at (5,1.5) {$-$};
    \node[anchor=west] at (5,1.5) {$c''$};
    \node at (5,0.5) {$-$};
    \node[anchor=west] at (5,0.5) {$c'$};
  \end{tikzpicture}
  \caption{The space $\M(A,|\cdot |_0)$ is parametrized as follows: for each prime ideal $\frakp \in \Spec(A)$ and $c \in [0,1]$, one obtains a multiplicative seminorm $| \cdot |_{\frakp,c} \coloneqq c^{\ord_{\frakp}(\cdot )}$ on $A$, where $\ord_{\frakp}$ denotes the discrete valuation on $A_{\frakp}$. (For $c \in \{ 0,1 \}$, this expression is in the limiting sense; said differently, $| \cdot |_{\frakp,1} = |\cdot |_0$ and $|\cdot |_{\frakp,0} = | \cdot \bmod \frakp |_0$ for all $\frakp \in \Spec(A)$.)}
  \label{fig:dedekind}
\end{figure}
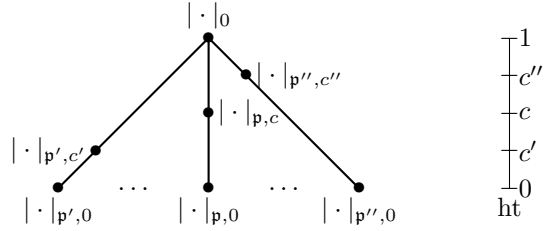

For $t \in (0,1]$ and $z \in \mathcal{M}(A)$, 
set $r_{t}(z) \in \M(A)$ to be the seminorm given by
$$
a \mapsto |a(r_t(z))| \coloneqq |a(z)|^{\min\{ 1, \frac{\mathrm{ht}(z)}{t} \}}.
$$
Set $r_0$ to be the identity map. The resulting function $\beta \colon [0,1] \times \M(A) \to \M(A)$, given by $(t,z) \mapsto r_t(z)$, is continuous and its fixed point locus is easily described: $r_t(z) = z$ if and only if $\mathrm{ht}(z) \geq t$. 
In fact, $\beta$ defines a strong deformation retraction of $\mathcal{M}(A)$ onto the trivial norm.

Now, if $X = \A^n_A$ is the $n$-dimensional affine space over $A$, 
then the arguments of~\cref{section_41} and~\cref{section_42} show that $\lambda^{-1}(| \cdot |_0) \hookrightarrow X^{\An}$ is a homotopy equivalence, after replacing $\beta \colon [0,1] \times X^{\An} \to X^{\An}$ with the function described above. 

\subsection{Proof of~\cref{theorem 3}}
Consider the ring of integers $A$ of a number field $K = \mathrm{Frac}(A)$.
It can be equipped with the norm
$$
\| \cdot \|_A \coloneqq \max_{\sigma} |\sigma(\cdot)|_{\infty},
$$
where the maximum ranges over all complex embedding $\sigma$ of $K$, and $|\cdot |_{\infty}$ is the Archimedean norm on $\mathbf{C}$.
For each such embedding $\sigma$, write $\| \cdot \|_{\hyb,\sigma} \coloneqq \max\{ |\sigma(\cdot)|_{\infty},|\cdot |_0 \}$ for the induced hybrid norm on $K$.

The spectrum $\M(A,\| \cdot \|_A)$ is well-understood (see e.g.~\cite[\S 1.4.1]{berkovich}) and it can be easily described in terms of a closed cover as follows:
\begin{itemize}
    \item the identity map $(A, \| \cdot \|_A) \to (A, |\cdot |_0)$ is contractive, and it induces a closed embedding $\M(A,|\cdot |_0) \hookrightarrow \M(A,\| \cdot \|_A)$;
    \item for each complex embedding $\sigma$ of $K$, the inclusion $(A,\| \cdot \|_A) \hookrightarrow (K,\| \cdot \|_{\hyb,\sigma})$ is contractive, and it induces a closed embedding $\M(K,\| \cdot \|_{\hyb,\sigma}) \hookrightarrow \M(A,\| \cdot \|_A)$.
\end{itemize}
All of the above closed subspaces of $\M(A,\| \cdot \|_A)$ intersect only at the trivial norm $|\cdot |_0$, and they form a cover of $\M(A,\| \cdot \|_A)$.

The tools have now been developed to complete the proof of the contractibility of $\A^{1,\An}_{(A,\| \cdot \|_A)}$.
For each complex embedding $\sigma$, there is an injective, continuous map 
$$
\A^{1,\hyb}_{(K,\| \cdot \|_{\hyb,\sigma})} \hookrightarrow \A^{1,\An}_{(A,\|\cdot \|_{A})}.
$$
Let $H_{\mathrm{hyb},\sigma} \colon [0,1] \times \A^{1,\hyb}_{(K,\| \cdot \|_{\hyb,\sigma})} \to \A^{1,\hyb}_{(K,\| \cdot \|_{\hyb,\sigma})}$ be a strong deformation retraction of $\A^{1,\hyb}_{(K,\| \cdot \|_{\hyb,\sigma})}$ onto $\A^{1,\an}_{(K,| \cdot |_0)}$; such a map exists by~\cref{theorem 1}.
Define the homotopy $H \colon [0,1] \times \A^{1,\An}_{(A,\|\cdot \|_{A})} \to \A^{1,\An}_{(A,\|\cdot \|_{A})}$ 
by
$$
x \mapsto \begin{cases}
H_{\mathrm{hyb},\sigma}(t,x) & \lambda(x) \in \M(K,\| \cdot \|_{\mathrm{hyb},\sigma}),\\
x & \mathrm{otherwise.}
\end{cases}
$$
This is well-defined because $H_{\mathrm{hyb},\sigma}$ is a strong deformation retraction, and it is clear that it is continuous.
Moreover, $H$ is a strong deformation retraction of $\A^{1,\An}_{(A,\|\cdot \|_{A})}$ onto $\A^{1,\An}_{(A,|\cdot |_{0})}$. 

By the extension of~\cref{theorem 2} to Dedekind domains described in~\cref{extension_to_dedekind_domains}, the inclusion 
$$\A^{1,\an}_{(K,|\cdot|_0)} \hookrightarrow \A^{1,\An}_{(A,|\cdot |_{0})}
$$
is a homotopy equivalence; 
in particular, $\A^{1,\An}_{(A,|\cdot|_0)}$ is contractible. 
It follows that $\A^{1,\An}_{(A,\|\cdot \|_{A})}$ is also contractible, which completes the proof of~\cref{theorem 3}.

\newcommand{\etalchar}[1]{$^{#1}$}

\end{document}